\documentclass[11pt]{article}
\usepackage{xcolor}
\usepackage[square,authoryear]{natbib}
\usepackage{marticle}
\usepackage[pagewise]{lineno}

\usepackage[all]{xy}
\usepackage{amssymb,amsfonts,ulem}

\usepackage{stackengine}
\usepackage{scalerel}

\newtheoremstyle{obs}
{3pt}
{3pt}
{}
{}
{\bfseries}
{.}
{.5em}
{}
\theoremstyle{obs}
\newtheorem{remark}[theorem]{Remark}
\newtheorem{example}[theorem]{Example}

\usepackage{authblk}

\title{\bf A new perspective on symplectic integration of constrained mechanical systems via discretization maps }
\author[1]{Mar\'ia\ Barbero Li\~n\'an\thanks{m.barbero@upm.es}}
\author[2]{David\ Mart\'{\i}n de Diego\thanks{david.martin@icmat.es}}
\author[3]{Rodrigo T. Sato Mart{\'\i}n de Almagro\thanks{rodrigo.t.sato@fau.de}}
\affil[1]{\small Departamento de Matem\'atica Aplicada, Universidad Polit\'ecnica de Madrid, Av. Juan de Herrera 4, 28040 Madrid, Spain. }
\affil[2]{\small Instituto de Ciencias Matem\'aticas (CSIC-UAM-UC3M-UCM), C/Nicol\'as Cabrera 13-15, 28049 Madrid, Spain.}
\affil[3]{\small Friedrich-Alexander-Universit\"at Erlangen-N\"urnberg, Institute of Applied Dynamics, Immerwahrstrasse 1, 91058 Erlangen, Germany}

\date{\today}

\begin{document}

\maketitle
\begin{abstract}
  A new geometric procedure to construct symplectic methods for constrained mechanical systems is developed in this paper. The definition of a map coming from the notion of retraction maps allows to adapt the continuous problem to the discretization rule rather than viceversa. As a result, the constraint submanifold is exactly preserved by the symplectic discrete flow and the extension of these methods to the case of non-linear configuration spaces is doable.

    \vspace{2mm}


    \vspace{2mm}
    

\end{abstract}

\section{Introduction}\label{Sec:Intro}

Nowadays, an interesting research line in the area of numerical simulation of mechanical systems consists of preserving simultaneously the qualitative and quantitative behaviour of the system so that associated underlying geometrical structures remain invariant. Such methods are called geometric integrators, see for instance \cite{hairer,sanz-serna,blanes}. They could preserve the energy behavior, some constants of motion, different geometrical structures (symplectic, Poisson, contact, Dirac...). Moreover, one of the essential goals of geometric integration is to produce methods exactly preserving the configuration space~\citep{LeRe}. 

In many mechanical simulations, including recent applications to accelerated optimization \citep{Jo18}, the configuration manifold is defined by the vanishing of one or several constraints. Such constraints are called holonomic constraints. Typically, they are discretized using a  procedure similar to the one used for discretizing the original mechanical system (Hamiltonian, Lagrangian, etc) without constraints. However, the resultant methods usually do not  exactly preserve the constraint submanifold, what leads to numerical instabilities as mentioned in~\cite{Danger}. That is why the propagation of errors in numerical algorithms for constrained systems is more complicated to analyze than for unconstrained systems \citep{LeimSkeel96,LeRe}.

In the recent paper by~\cite{21MBLDMdD} the notion of retraction map~\citep{AbMaSeBookRetraction} is used to define a discretization map $R^h_d: TQ\rightarrow Q\times Q$
that induces by lifting a  symplectic integrators for Lagrangian and Hamiltonian systems.
 For instance, these discretization maps were used to derive discrete Lagrangians from a continuous Lagrangian
 $L: TQ\rightarrow {\mathbb R}$ as follows: 
\[
L_d(q_0, q_1)=h L\left((R^h_d)^{-1}(q_0, q_1)\right),
\] according to the seminal paper on discrete variational calculus by~\cite{MW_Acta}.  
In an Euclidean space, the above-mentioned discretization maps are usually given by
\begin{equation}\label{eq:EuclIntro}
R^h_d(q, v)=(q-h\, \alpha\, v, q+h (1-\alpha) v),\quad 0\leq \alpha\leq 1\, ,
\end{equation}
that define the following discrete Lagrangian function:
\[
L_d (q_0, q_1)=h\, L\left( (1-\alpha)q_0+\alpha q_1, \frac{q_1-q_0}{h}\right)\,.
\]
If the continuous system is additionally subjected to holonomic  constraints $\phi(q)=0$, a natural option is to evaluate the constraints not at the discrete points $\{q_k\}$ of the discrete trajectory, but at some intermediate points as follows: 
\begin{equation}\label{eq:phi}
    \phi\left(\tau_Q((R^h_d)^{-1}(q_0, q_1))\right)=0\, ,
\end{equation}
where $\tau_Q: TQ\rightarrow Q$ is the canonical tangent bundle projection. For the example in~\eqref{eq:EuclIntro} the Equation~\eqref{eq:phi} becomes
\begin{equation*} 
\phi ((1-\alpha)q_0+\alpha q_1)=0\, .
\end{equation*}

Consequently, as mentioned in~\cite{Danger}, the holonomic constraints are not preserved at points of the discrete trajectory (step values), but at some intermediate values (stage values). That is why numerical instabilities may appear when solving the equations. Recently, there has been an increasing interest in developing numerical integrators for holonomic dynamics focusing on reducing the error in the holonomic constraints, as well as in the tangency condition so that the discrete trajectories stay close to the constraint submanifold~\citep{2022CelledoniEtAl,Leok2021Constrained,2022Joris,2022LeokSOn}.
In this paper, an alternative procedure to prevent that problem is described so that the holonomic constraints are satisfied at all discrete points.

The proposed approach is important because it allows for the systematic and constructive study   of relevant families of holonomic systems used, for instance, to model robots made by discrete links or multibody systems as described in Section~\ref{Sec:multibody}.

The contents are organized as follows. Section~\ref{Sec:constrained} includes the definition of constrained mechanical systems. When the mechanical system is constrained to evolve on a submanifold of the configuration space, the equations can be written on the submanifold or on the entire configuration space to end up projecting the equations to the constraint submanifold using holonomic constraints. The former approach is referred in this paper as ``intrinsic"  and the latter as ``extrinsic". However, we prove that both dynamics are equivalent in Proposition~\ref{prop:extr=intrin}. In Section~\ref{Sec:SymplHolon} a symplectic integrator for holonomic mechanical systems is obtained using the theory on discretization maps developed in~\citep{21MBLDMdD} and reviewed in Section~\ref{Sec:Rd}. In order to preserve exactly the constraint submanifold {\sl the continuous problem is slightly perturbed} (see Definition \ref{definition-perturbed}) so that it fits perfectly into the discretization map. Examples of symplectic integrators for mechanical Lagrangian systems on Euclidean spaces and on Lie groups are given in Section~\ref{Sec:ExampleEuclid}. The theory is developed in the Lagrangian framework on a general manifold and in the Hamiltonian framework on Lie groups for constrained systems in Section~\ref{section:Lie} using the trivialization of the needed manifolds.
The Appendix recalls the dynamics of Lagrangian unconstrained mechanical systems on Lie groups using the notion of trivialized tangent map used in~\citep{BouMa} so that the paper is self-contained and Section~\ref{section:Lie} can be easier to read. Finally, conclusions and future work are included in Section~\ref{S:future}.

\section{Constrained mechanical systems} \label{Sec:constrained}

Let us introduce the elements necessary to describe a constrained mechanical system. 
Consider a submanifold $N$ of a finite dimensional manifold $Q$ where $i_N\colon N\hookrightarrow Q$ is the canonical inclusion. Typically, $Q={\mathbb R}^m$ is considered, even though it is not the most general situation as described in Sections~\ref{Sec:multibody} and ~\ref{section:Lie}. 
The dynamics on $N$ is specified by a Lagrangian function $L: TN\rightarrow {\mathbb R}$ that is assumed to be regular.

At this point we could work on the submanifold $N$ from an intrinsic viewpoint or we could embed the problem in $Q$ to obtain the equations of motion as a constrained system. That family of constrained systems are called holonomic systems, as opposed to nonholonomic systems where constraints also involve velocities.

\subsection{Intrinsic viewpoint}\label{sec:intrinsic}

When working on the manifold $N$, the equations of motion are completely determined using Hamilton's principle \citep{AbMa}. This principle states that the solution curves $x: I\subseteq {\mathbb R}\rightarrow N$ of the Lagrangian system determined by $L\colon TN \rightarrow \mathbb{R}$ are the critical points of the functional: 
\[
J=\int^T_0 L(x(t), \dot{x}(t))\; {\rm d}t\, .
\]
In local coordinates $(x^a)$, $1\leq a\leq n=\dim N$, for $N$ and induced coordinates $(x^a, \dot{x}^a)$ for $TN$, the solution curves verify the Euler-Lagrange equations: 
\begin{equation} \label{eq:EL eq}
\frac{d}{dt}\left( \frac{\partial L}{\partial \dot{x}^a}\right)-\frac{\partial L}{\partial {x}^a}=0\, .
\end{equation}
Under  regularity assumption on the Lagrangian, the solution curves can also be described by the following Hamilton's equations: 
\begin{equation} \label{eq:H eq}
\dot{x}^a=\frac{\partial H}{\partial p_a}\; ,\qquad  \dot{p}_a=-\frac{\partial H}{\partial x^a}\, ,
\end{equation}
using the coordinates $(x^a, p_a)$ for the cotangent bundle $T^*N$ of $N$ and where $H: T^*N\rightarrow \mathbb{R}$ is the Hamiltonian function given by the Legendre transform of the energy function determined by $L$ (see~\cite{AbMa} for details).

Standard symplectic integrators \citep{sanz-serna,hairer,blanes} can be used to numerically integrate Euler-Lagrange equations~\eqref{eq:EL eq} or Hamilton's equations~\eqref{eq:H eq}. However, the main difficulty is to deal with a nonlinear configuration space to define symplectic integrators (see for instance \citep{BogMa} and references therein  for the case of Lie groups). The main objective of the paper is to elucidate the geometry of symplectic integration of holonomic systems and to describe methods valid for general  manifolds, not only for Euclidean spaces.

\subsection{Extrinsic viewpoint}\label{sec:extrinsic}

To overcome the problem of working with nonlinear spaces, or spaces which do not admit simple discretizations, it is common to introduce holonomic constraints. Here, the manifold $Q$ is not assumed to be a linear space. The motion is restricted to remain in a $n$-dimensional submanifold $N$ of $Q$ determined by some algebraic equations on the position variables of the
system such as
$$\phi^{\alpha}(q)=0, \quad 1\leq \alpha\leq m-n\, ,$$ 
where $\phi^{\alpha}$ are smooth functions and $\dim Q=m>n$ (see \citep{arnold, LeRe}).

The equations on $Q$, with local coordinates $(q^i)$, $1\leq i\leq n$, of a mechanical system determined by a Lagrangian function $L: TN\rightarrow {\mathbb R}$ and holonomic constraints are given by
\begin{eqnarray*}
\frac{d}{dt}\left(
\frac{\partial \tilde{L}}{\partial \dot{q}^i}
\right)-\frac{\partial \tilde{L}}{\partial q^i}
&=&-\Lambda_{\alpha}\frac{\partial\phi^{\alpha}}{\partial q^i},\; \quad \mbox{for }i=1,\dots,m,\\
\phi^{\alpha}(q)&=&0\; ,
\end{eqnarray*}
where $\Lambda_\alpha$ are functions on $TQ$ that play the role of Lagrange multipliers to be determined and $\tilde{L}: TQ\rightarrow {\mathbb R}$ is an arbitrary extension of the Lagrangian $L: TN \rightarrow {\mathbb R}$ such that $\tilde{L}_{|TN}=L$. In the Hamiltonian framework, for a Hamiltonian function $H\colon T^*Q\rightarrow \mathbb{R}$ we obtain the following equations: 
\begin{eqnarray*}
\dot{q}^i&=&\frac{\partial H}{\partial p_i}\; , \\ \dot{p}_i&=&-\frac{\partial H}{\partial q^i}-\Lambda_{\alpha}\frac{\partial \phi^\alpha}{\partial q^i}\; ,\\
\phi^{\alpha}(q)&=&0\; ,
\end{eqnarray*}
where $(q^i, p_i)$ are now the coordinates  of the cotangent bundle $T^*Q$ of $Q$. Note that the momenta in this approach live in a submanifold of $T^*Q$ (in fact $Leg_L(TQ)$, see \cite{MW_Acta}), while in the intrinsic framework the momenta live in $T^*N$.

\subsection{Geometric relation between both approaches}\label{Sec:extr-intr}
Now, the well-known equivalence between the two approaches in Sections~\ref{sec:intrinsic} and~\ref{sec:extrinsic} is geometrically justified. This reasoning is useful to describe with generality the construction of symplectic integrators for holonomic systems in the following sections. 

Let $\tau_Q: TQ\rightarrow Q$, $\tau_N: TN\rightarrow N$, $\pi_Q: T^*Q\rightarrow Q$ and  $\pi_N: T^*N\rightarrow N$ be the corresponding canonical projections of the tangent and cotangent bundles. Remember that any cotangent bundle $T^*M$ is a symplectic manifold $(T^*M,\omega_M)$ with the canonical symplectic structure $\omega_M$, i.e. a closed non-degenerate 2-form, locally given by ${\rm d}p_i\wedge{\rm d}q^i$, in local coordinates $(q^i,p_i)$ for $T^*M$.

For a Lagrangian function $L:TN \rightarrow {\mathbb R}$ we can define two natural Lagrangian submanifolds (see for details  \citep{Tu, 13GuiStern}) associated with the viewpoints in Sections~\ref{sec:intrinsic} and~\ref{sec:extrinsic}, respectively:
\begin{itemize}
\item the Lagrangian submanifold $\hbox{Im} \, {\rm d} L={\rm d} L(TN)$ of the symplectic manifold $(T^*TN, \omega_{TN})$ for the intrinsic point of view. Observe that $\dim(\hbox{Im} \, {\rm d} L)=2\dim N$.
\item the Lagrangian submanifold 

\begin{equation}\label{aqr}
\begin{array}{lcl}
\Sigma_{L}&\!=\!&\left\{\mu_{i_{TN}(u_x)}\in T^*TQ\; |\; (i_{TN})^*\,\mu_{i_{TN}(u_x)}={\rm d}{L}(u_x), \quad u_x\in T_xN\right\}\\
\end{array}
\end{equation}
of the symplectic manifold $(T^*TQ, \omega_{TQ})$ for the extrinsic viewpoint, where $i_{TN}: TN\hookrightarrow TQ$ is the canonical inclusion given by the tangent map of the inclusion $i_N$, that is,  $i_{TN}={\rm T}i_N$ and $\left(i_{TN}\right)^*\colon T^*TQ\rightarrow T^*TN$ is the corresponding pull-back of $i_{TN}$. Observe that $\dim\Sigma_L=2\dim Q$. 
\end{itemize}
 Denote by  $(TT N)^0$  the annihilator of $TT N$ whose fibers at $u_x \in TN$ are given by 
 \begin{equation}\label{eq:annih}(TT N)^0_{{u_x}}=\left(T_{{u_x}}TN\right)^0=\left\{\Lambda\in {T^*_{i_{TN}(u_x)}TQ\, |\, \langle \Lambda, {\rm T}i_{TN}\left(X\right)\rangle} =0, \;\forall \; X\in T_{{u_x}}TN\right\} \end{equation}
From now on,  for the sake of notational simplicity we identify $v_q = i_{TN}(u_x)$ (with $q=i_N(x)$) and start to use $v_q$ in place of $u_x$. The context will make it clear if this is to be understood as a point in $TN$ or in $TQ$. The Lagrangian submanifold $\Sigma_L$ can alternatively be described by Lagrange multipliers and an extension function $\tilde{L}: TQ\rightarrow {\mathbb R}$ of $L: TN\rightarrow {\mathbb R}$ such that $\tilde{L}_{|TN}=L$ as follows:
\begin{align}\label{aqr1}
\Sigma_L&=\,\left\{\mu_{v_q}\in T^*TQ\; |\; \mu_{v_q}-{\rm d}{\tilde{L}}(v_q) \in (T_{v_q} i_{TN}(T N))^0\right\}\\
&=\,\left\{\mu_{v_q}\in T^*TQ\, |\, \mu_{v_q}={\rm d}\tilde{L}(v_q)+\lambda_{\alpha}{\rm d}\phi^{\alpha}(q)+ \tilde{\lambda}_{\alpha}{\rm d}\left({\rm d}_T\phi^{\alpha}(v_q)\right),  v_q\in T_qQ\right\}\, ,\nonumber
\end{align}
where ${\rm d}_T\phi^{\alpha}(v_q)=\frac{\rm d}{{\rm d}t}\Big|_{t=0} (\phi^{\alpha}(\sigma(t)))$ for any curve $\sigma: I\rightarrow Q$ such that $\sigma(0)=q$ and $\dot{\sigma}(0)=v_q$. This represents a tangency condition on vectors to remain on $TN$ and it is sometimes referred to as `constraint at velocity level'.

In local coordinates:
\[
{\rm d}_T\phi^{\alpha}(q,\dot{q})=\frac{\partial \phi^{\alpha}}{\partial q^i}(q)\, \dot{q}^i\, .
\] 
The constraints with the Lagrange multipliers in Equation~\eqref{aqr1} precisely describe the tangent bundle $TN$ as a submanifold of $TQ$:
\[
TN=\left\{ v_q \in TQ\; |\; 
\phi^{\alpha}(q)=0, \quad 
{\rm d}_T\phi^{\alpha}(v_q)=0\right\}\,.
\]
Hence, once a Lagrangian on $TN$ is given, the two Lagrangian submanifolds $\Sigma_L$ and ${\rm d}L(TN)$ describe the same solutions on $TN$ as the following diagram shows: 
\begin{equation}\label{Diagram}
\centerline{\xymatrix{  T^*TN \supseteq {\rm d}L(TN) \ar[d]^{\pi_{TN}}   && \Sigma_L\ar@{^(->}[rr]\ar[ll]_{\quad (i_{TN})^*}\ar[d]^{(\pi_{TQ})_{|TN}}&& T^*TQ\ar[d]^{\pi_{TQ}} \\
TN\ar[rr]^{i_{TN}} &&  i_{TN}(TN)\equiv  TN\ar@{^(->}[rr]^{i_{TN}}&&TQ
}}\end{equation}
\noindent where $\pi_{TN}: T^*TN\rightarrow TN$ and $\pi_{TQ}: T^*TQ\rightarrow TQ$ are the canonical projections for the corresponding cotangent bundles. 

Now we want to show a geometrical interpretation of the equivalence of the dynamics with holonomic constraints  described in Sections~\ref{sec:intrinsic} and~\ref{sec:extrinsic}.

For this purpose it is necessary to use the canonical symplectomorphism $\alpha_Q: TT^*Q\rightarrow T^*TQ$ between double vector bundles \citep{TuHamilton,1999TuUr}.
Locally,
\begin{equation}\label{Eq:TuSymplec}
\begin{array}{crcl} \alpha_Q\colon & TT^*Q & \longrightarrow &  T^*TQ\\&
(q,p,\dot{q},\dot{p}) & \longrightarrow & (q,\dot{q},\dot{p},p).
\end{array}
\end{equation} As described in~\citep{TuHamilton}, the symplectomorphism $\alpha_Q$ is between the sympletic manifold $(TT^*Q, {\rm d}_T\omega_Q)$ and the natural symplectic manifold $(T^*TQ, \omega_{TQ})$. Recall that in local coordinates $(q,p,\dot{q},\dot{p})$ for $TT^*Q$, the symplectic form ${\rm d}_T\omega_Q$ has the following expression: ${\rm d}_T \omega_Q= {\rm d}q\wedge {\rm d}\dot{p}+{\rm d}\dot{q}\wedge {\rm d}p$.

We need to use 
the Lagrangian submanifolds
$(\alpha_Q)^{-1} (\Sigma_L)$ and 
$(\alpha_N)^{-1} ({\rm d}L(TN))$ of the symplectic manifolds
$(TT^*Q,{\rm d}_T\omega_Q)$ and
$(TT^*N,{\rm d}_T\omega_N)$, respectively.

\begin{proposition}\label{propo-equiv} Let
 $N$ be a submanifold of $Q$ and $L\colon TN \rightarrow \mathbb{R}$ be a regular Lagrangian function. Then,
\[
(\alpha_Q)^{-1} (\Sigma_L)=\left\{
X_{\mu_q}\in T_{\mu_q}T^*Q\; |\;  \left( {\rm T}_{\mu_q}i_N^*\right) (X_{\mu_q})\in (\alpha_N)^{-1} ({\rm d}L(TN))\right\}\, .
\]
\end{proposition}
\begin{proof}
The proposition is a direct consequence of the commutativity  of the following diagram: 
\

\centerline{\xymatrix{ T^*TQ\supseteq T^*_{TN}TQ =\pi_{TQ}^{-1}(TN) \ar[d]^{i^*_{TN}} \ar[rr]^{\alpha^{-1}_Q} &&T(T^*_NQ)=T(\pi_Q^{-1}(N))\subset TT^*Q
\ar[d]^{Ti^*_N} \\
 T^*TN   \ar[rr]^{\alpha^{-1}_N} && TT^*N
}}
\noindent where $\pi_{TQ}: T^*TQ\rightarrow TQ$ is the canonical projection of the cotangent bundle. 
Diagram~\eqref{Diagram} shows that
\[
i_{TN}^*(\Sigma_L)= {\rm d}L(TN)\,
\]
and $\Sigma_L$ fibers onto $TN$. Thus, \[
Ti_N^*\left((\alpha_Q)^{-1} (\Sigma_L)\right)=
(\alpha_N)^{-1} ({\rm d}L(TN))\, .\qedhere
\]
\end{proof}
It follows that the dynamics given by $(\alpha_N)^{-1} ({\rm d}L(TN))$ and $(\alpha_Q)^{-1} (\Sigma_L)$ , both Lagrangian submanifolds, are related. 
However, the intrinsic description given by $dL(TN)$ does not uniquely determine the Lagrange multipliers of the extrinsic description given by $\Sigma_L$, as shown in the proof below. 

\begin{proposition}\label{prop:extr=intrin}
The intrinsic dynamics given in Subsection \ref{sec:intrinsic} and the  dynamics with holonomic constraints given in Subsection \ref{sec:extrinsic} are equivalent. 

\end{proposition}
\begin{proof} Proposition~\ref{propo-equiv} describes the intrinsic description in Section~\ref{sec:intrinsic} with the Lagrangian submanifold $\Sigma_L$ locally described by introducing Lagrange multipliers and an extension $\tilde{L}: TQ\rightarrow {\mathbb R}$ of the Lagrangian function $L: TN\rightarrow {\mathbb R}$ as follows:
\[
\Sigma_L=\{\mu_{v_q}={\rm d}\tilde{L}(v_q)+\lambda_{\alpha}{\rm d}\phi^{\alpha}(q)+ \tilde{\lambda}_{\alpha}{\rm d}({\rm d}_T\phi^{\alpha})(v_q), \quad v_q\in T_qN\} \, .
\]
In coordinates, the points of $\Sigma_L$ are given by
\begin{equation*}
\left(q^i, \dot{q}^i;  
\frac{\partial \tilde{L}}{\partial q^i}
+\lambda_{\alpha}\frac{\partial\phi^{\alpha}}{\partial q^i}+ \tilde{\lambda}_{\alpha}\frac{\partial^2\phi^{\alpha}}{\partial q^i\partial q^j}\dot{q}^j, 
\frac{\partial \tilde{L}}{\partial \dot{q}^i}
+\tilde{\lambda}_{\alpha}\frac{\partial \phi^{\alpha}}{\partial q^i}\right)
\end{equation*}
restricted to  $\phi^{\alpha}(q)=0$ and ${\rm d}_T\phi^{\alpha}(q, \dot{q})=0$. Remember that the constraint determined by ${\rm d}_T\phi^{\alpha}$ is a dynamical consequence of $\phi^{\alpha}$ so that the solution curves starting on $N$ remain on $N$. 
Therefore, the points in
$\alpha_Q^{-1}(\Sigma_L)$ have the following local expression 
\begin{equation*}
\left(q^i, p_i=\frac{\partial \tilde{L}}{\partial \dot{q}^i}
+\tilde{\lambda}_{\alpha}\frac{\partial \phi^{\alpha}}{\partial q^i}; \dot{q}^i, 
\dot{p}_i=
\frac{\partial \tilde{L}}{\partial q^i}
+\lambda_{\alpha}\frac{\partial\phi^{\alpha}}{\partial q^i}+ \tilde{\lambda}_{\alpha}\frac{\partial^2\phi^{\alpha}}{\partial q^i\partial q^j}\dot{q}^j\right)\,,
\end{equation*}
and they satisfy $\phi^{\alpha}(q)=0$ and ${\rm d}_T\phi^{\alpha}(q, \dot{q})=0$. Thus, considering $\alpha_Q^{-1}(\Sigma_L)$ as an implicit system of differential equations on $TT^*Q$ we obtain:
\begin{equation*}\label{eq:EL-N}
\frac{d}{dt}\left( 
\frac{\partial \tilde{L}}{\partial \dot{q}^i}
+\tilde{\lambda}_{\alpha}\frac{\partial \phi^{\alpha}}{\partial q^i}\right)
=
\frac{\partial \tilde{L}}{\partial q^i}
+\lambda_{\alpha}\frac{\partial\phi^{\alpha}}{\partial q^i}+ \tilde{\lambda}_{\alpha}\frac{\partial^2\phi^{\alpha}}{\partial q^i\partial q^j}\dot{q}^j\, .
\end{equation*}
After some straightforward computations  we obtain that the dynamics determined by $\alpha_Q^{-1}(\Sigma_L)$ as a submanifold of $TT^*Q$ leads to the extrinsic description of the holonomic dynamics in Section~\ref{sec:extrinsic}:
\begin{eqnarray*}
\frac{\rm d}{{\rm d}t}\left(
\frac{\partial \tilde{L}}{\partial \dot{q}^i}
\right)-\frac{\partial \tilde{L}}{\partial q^i}
&=&\left(\lambda_{\alpha}-\frac{\rm d}{{\rm d}t}(\tilde{\lambda}_{\alpha})\right)\frac{\partial\phi^{\alpha}}{\partial q^i}=
-\Lambda_{\alpha}\frac{\partial\phi^{\alpha}}{\partial q^i}\; ,\label{eq:proof_extrinsic}\\
\phi^{\alpha}(q)&=&0\; .\nonumber
\end{eqnarray*}

It follows that the Lagrange multipliers $\lambda_\alpha$ and $\tilde{\lambda}_\alpha$ are not uniquely determined. Nevertheless, the solutions on $TN$ for both equations are the same and we conclude that the dynamics are equivalent. \qedhere
\end{proof}

We emphasize the fact that solution curves to the Euler-Lagrange equations on $N$ must remain on the submanifold $N$. That is why the additional constraint ${\rm d}_T\phi^{\alpha}$ appears when considering the problem on the whole configuration manifold and two families of Lagrange multipliers are necessary. This is a crucial point when discretizing constrained mechanical systems to obtain numerical integrators that simultaneously preserve the holonomic constraints and the tangency condition.

\section{Discretization maps} \label{Sec:Rd}

We briefly introduce the notion of discretization map developed in the recent work\ \citep{21MBLDMdD}. Such a map is obtained by using retraction maps, first introduced in the literature in~\citep{1931Borsuk} and later used in optimization problems on matrix groups in~\citep{2002Adler}. Let us now focus on the notion of retraction map as defined in~\citep{AbMaSeBookRetraction}. A retraction map on a manifold $M$ is a smooth map 
$R\colon U\subseteq TM \rightarrow M$ where $U$ is an open subset containing the zero section  of the tangent bundle 
such that the restriction map $R_x=R_{|T_xM}\colon T_xM\rightarrow M$ satisfies
\begin{enumerate}
	\item $R_x(0_x)=x$ for all $x\in M$,

	\item 
	$T_{0_x}R_x={\rm Id}_{T_xM}$ with the identification $T_{0_x}T_xM\simeq T_xM$.
\end{enumerate}

\begin{example}
If  $(M, g)$ is a Riemannian manifold, then the exponential map $\hbox{exp}^g: U\subset TM\rightarrow M$ is a typical example of retraction map: 
$
\hbox{exp}^g_x(v_x)=\gamma_{v_x}(1), 
$
where $\gamma_{v_x}$ is the unique  Riemannian geodesic satisfying $\gamma_{v_x}(0)=x$ and $\gamma'_{v_x}(0)=v_x$~\citep{doCarmo}. 
\end{example}

 We are more interested in the discretization of the manifold where the dynamics takes place than in the discretization of the variational principle as developed in~\citep{MW_Acta}. Having that in mind, in \citep {21MBLDMdD}  retraction maps have been used to define discretization maps $R_d\colon U \subset TM  \rightarrow M\times M$, where  $U$ is an open neighbourhood of the zero section of $TM$, 
\begin{eqnarray*}
	R_d\colon U \subset TM & \longrightarrow & M\times M\\
	v_x & \longmapsto & (R^1(v_x),R^2(v_x))\, .
\end{eqnarray*}
Discretization maps satisfy the following properties:
\begin{enumerate}
	\item $R_d(0_x)=(x,x)$, for all $x\in M$.
	\item $T_{0_x}R^2_x-T_{0_x}R^1_x={\rm Id}_{T_xM} \colon T_{0_x}T_xM\simeq T_xM \rightarrow T_xM$ is equal to the identity map on $T_xM$ for any $x$ in $M$.
\end{enumerate}
Thus, the discretization map $R_d$ is a local diffeomorphism.

We introduce here the inversion map so that higher-order methods can be obtained by composition of Lagrangian submanifolds as described in Section~\ref{Sec:Composition}. The inversion map $I_M: M\times M\rightarrow M\times M$ is given by $I_M (x,y)=(y, x)$ for all $x, y\in M$.
Now, if  $R_d:U\subset TM\rightarrow M\times M$ is a discretization map, then 
$R^*_d: \overline{U}\subset TM\rightarrow M\times M$ with $\overline{U}=\{ v_x\in TM\; |\, -v_x\in U\}$ and defined by 
\[
R^*_d(v_x)=\left(I_M \circ R_d\right)(-v_x)
\]
is also a discretization map called the adjoint discretization map of $R_d$. 
We will say that a discretization map is symmetric if $R^*_d=R_d$.

\begin{example}
Some examples of discretization maps on Euclidean vector spaces are:
\begin{itemize}
	\item Explicit Euler method:  $R_d(x,v)=(x,x+v)$ and its adjoint discretization map $R^*_d(x,v)=(x-v,x)$.
	\item Midpoint rule:  $R_d(x,v)=\left( x-\dfrac{v}{2}, x+\dfrac{v}{2}\right).$ It is a symmetric discretization map.
		\item $\theta$-methods with $\theta\in [0,1]$:  \hspace{3mm} $R_d(x,v)=\left( x-\theta \, v, x+ (1-\theta)\, v\right).$
\end{itemize}
\end{example}

The manifold $M$ in the definition of the discretization map can be replaced by the cotangent bundle $T^*Q$ where the integral curves of Hamilton's equations live. The objective is to define a discretization map on $T^*Q$ to be used to define geometric integrators for Hamilton's equations. A useful one is obtained by cotangently lifting a discretization map $R_d\colon TQ\rightarrow Q\times Q$ on $Q$ in such a way that the resulting map $R^{T^*}_d\colon TT^*Q \rightarrow T^*Q\times T^*Q$ is a symplectomorphism and symplectic integrators can be defined. The following three symplectomorphisms are needed for such a construction (see \cite{21MBLDMdD} for more details): 
\begin{itemize}
\item The cotangent lift of a diffeomorphism $F: M_1\rightarrow M_2$ which is  the symplectomorphism defined by:
	\begin{equation*} 
	\hat{F}: T^*M_1 \longrightarrow  T^*M_2 \mbox{ such that } 
\hat{F}=(TF^{-1})^*.
	\end{equation*}
	\item The canonical symplectomorphism introduced in~\eqref{Eq:TuSymplec}:
	\begin{equation*} \alpha_Q\colon TT^*Q  \longrightarrow  T^*TQ  \mbox{ locally given by } \alpha_Q(q,p,\dot{q},\dot{p})= (q,\dot{q}, \dot{p}, p).
	\end{equation*}

	\item  The symplectomorphism between $(T^*(Q\times Q), \omega_{Q\times Q})$     and   
	$(T^*Q\times T^*Q, \Omega_{12}={\rm pr}_2^*\omega_Q-{\rm pr}^*_1\omega_Q)$, where ${\rm pr}_i\colon T^*Q\times T^*Q\rightarrow T^*Q$ is the projection to the $i$-th factor:
		\begin{equation}\label{eq:symplecticPhi}
	\Phi:T^*Q\times T^*Q \longrightarrow T^*(Q\times Q)\; , \; 
	\Phi(q_0, p_0; q_1, p_1)=(q_0, q_1, -p_0, p_1).	\end{equation}
	\end{itemize}
The following commutative diagram summarizes the construction process from $R_d$ to $R_d^{T^*}=\Phi^{-1}\circ\widehat{R_d}\circ \alpha_Q$:
	\begin{equation}\label{Diagram:cotangentlift}
\xymatrix{ {{TT^*Q }} \ar[rr]^{{{R_d^{T^*}}}}\ar[d]_{\alpha_{Q}} && {{T^*Q\times T^*Q }}  \\ T^*TQ \ar[d]_{\pi_{TQ}}\ar[rr]^{	\widehat{R_d}}&& T^*(Q\times Q)\ar[u]_{\Phi^{-1}}\ar[d]^{\pi_{Q\times Q}}\\ TQ \ar[rr]^{R_d} && Q\times Q }
\end{equation}

\begin{proposition}\citep{21MBLDMdD}
	If $R_d\colon TQ\rightarrow Q\times Q$ is a discretization map on $Q$, then \begin{equation}\label{eq:cotLiftdiscrete}{{R_d^{T^*}=\Phi^{-1}\circ \widehat{R_d}\circ \alpha_Q\colon TT^*Q\rightarrow T^*Q\times T^*Q}}\end{equation}
is a discretization map on $T^*Q$.\label{Prop:above}
\end{proposition}

Consequently, the  discretization map $R_d^{T^*}$ is the right tool to construct symplectic integrators for mechanical systems since it has been obtained by composition of symplectomorphisms. 
 
\begin{corollary}\citep{21MBLDMdD} The discretization map
 ${{R_d^{T^*}}}$ in Proposition~\ref{Prop:above} is a symplectomorphism between $(T(T^*Q), {\rm d}_T \omega_Q)$ and $(T^*Q\times T^*Q, \Omega_{12})$.
\end{corollary}

\begin{example}\label{example3} On $Q={\mathbb R}^n$ the midpoint symmetric discretization map 
	$R_d(q,v)=\left(q-\frac{1}{2}v, q+\frac{1}{2}v\right)$ is cotangently lifted to the symplectomorphism
		$$R_d^{T^*}(q,p,\dot{q},\dot{p})=\left( q-\dfrac{1}{2}\,\dot{q}, p-\dfrac{\dot{p}}{2}; \; q+\dfrac{1}{2}\, \dot{q}, p+\dfrac{\dot{p}}{2}\right)\, .$$
\end{example}

\section{Symplectic integration of holonomic mechanical systems} \label{Sec:SymplHolon}

We now adapt discretization maps to numerically integrate constrained mechanical systems, in particular, the holonomic ones described in Section~\ref{Sec:constrained}. Consider a holonomic system given by a Lagrangian $L: TN\rightarrow {\mathbb R}$, where $N$ is an embedded  submanifold of the ambient manifold $Q$. To obtain a symplectic integrator for such a constrained system the following two ingredients are necessary:
\begin{itemize}
    \item A  discretization map $R_d: TQ\rightarrow Q\times Q$ of the ambient space.
    \item An arbitrary extension $\tilde{L}: TQ\rightarrow {\mathbb R}$  of $L$ such that $\tilde{L}\big|_{TN}=L$.
\end{itemize}

The following results are satisfied by a pair of sufficiently close points $q_k, q_{k+1}$ in $Q$, that is, there exists an  open neighborhood $U$ of $Q\times Q$ containing $(q_k, q_{k+1})$ where $(R_d)^{-1}_{|U}$ is a local diffeomorphism. 

Our focus is to precisely preserve the holonomic constraints for the discrete flow (see ~\citep{2022Joris,2022LeokSOn} for alternative approaches). To guarantee that  is why the continuous holonomic mechanical system must be slightly perturbed.

\begin{definition}\label{definition-perturbed} Let $R_d$ be a discretization map on $Q$ and $L$ be a Lagrangian function on $N$. A {\bf modified constrained variational problem } is given by $\tilde{L}: TQ\rightarrow {\mathbb R}$ such that $\tilde{L}\big|_{TN}=L$ and the constraint submanifold of $TQ$ defined by: 
\begin{equation}\label{eq:TN_Rdv2}
	{\mathcal T}^{R_d}_hN=\left(R_d^h\right)^{-1}(N\times N)\, ,
	\end{equation}	
	where $R^h_d(v_q)=R_d(h\, v_q)$.
\end{definition}
The calligraphic ${\mathcal T}$ is used because the subset of $TQ$ defined in Equation~\eqref{eq:TN_Rdv2} is not necessarily the tangent bundle of $N$, neither does it project onto $N$ under the canonical tangent projection $\tau_Q$. However, it is the right subset to guarantee that the sequence of points generated by the numerical integrator to be obtained always satisfies the constraints and the tangency condition coming from the holonomic constraints. That is why in the continuous setting we have perturbed the original constrained mechanical system to obtain the {\bf modified constrained variational problem } $(\tilde{L},{\mathcal T}_h^{R_d}N)$. Similar to Equation~\eqref{aqr}, the pair $(\tilde{L},{\mathcal T}_h^{R_d}N)$ defines the following Lagrangian submanifold 
\begin{equation}\label{eq:motionConstrained}
\Sigma^d_{\tilde{L}}=\left\{\mu\in T^*TQ\, | \, \mu-{\rm d}\tilde{L}\in \left(T\left({\mathcal T}_h^{R_d}N\right)\right)^0\right\}
\end{equation}
of the symplectic manifold $(T^*TQ,\omega_{TQ})$, where $\left(T\left({\mathcal T}_h^{R_d}N\right)\right)^0$ denotes the annihilator of $T\left({\mathcal T}_h^{R_d}N\right)$ whose fibers are defined as in Equation~\eqref{eq:annih}.

Locally, if $N$ is determined by the vanishing of the constraints $\phi^{\alpha}(q)=0$, then ${\mathcal T}_h^{R_d}N$ is described by the vanishing of the constraint functions $\phi^\alpha_1(q,v)=\left(\phi^{\alpha}\circ \hbox{pr}_1\circ R^h_d\right)(q,v)=0$  and $\phi^\alpha_2(q,v)=\left(\phi^{\alpha}\circ \hbox{pr}_2\circ R^h_d\right)(q,v)=0$ defined on $TQ$, where $\hbox{pr}_l: Q\times Q\rightarrow Q$ is the projection onto the $l$-th factor with $l=1,2$. Thus, any $\mu\in \Sigma^d_{\tilde{L}}$ satisfies that $\pi_{TQ}(\mu)\in {\mathcal T}_h^{R_d}N$ and it is given by
\[
\mu=d\tilde{L}+\lambda_{\alpha}^1 d\phi_1^{\alpha}+\lambda_{\alpha}^2 d\phi_2^{\alpha}\, ,
\]
where $1\leq \alpha\leq m-n$, $\lambda_{\alpha}^1$ and $\lambda_{\alpha}^2$ are Lagrange multipliers to be determined. 

Once again using the Tulczyjew's diffeomorphism, the dynamics of the modified constrained variational problem can also be described by the Lagrangian submanifold  $$\tilde{S}_d=\alpha_Q^{-1}(\Sigma^d_{\tilde{L}})$$ of $(TT^*Q,{\rm d}_T\omega_Q)$.
Locally, the points in 
$\tilde{S}_d$ are given by
\begin{equation*}
\left(q^i, p_i=\frac{\partial \tilde{L}}{\partial \dot{q}^i}
+{\lambda}_{\alpha}^1\frac{\partial \phi_1^{\alpha}}{\partial \dot{q}^i}+\lambda^2_{\alpha}\frac{\partial \phi^{\alpha}_2}{\partial \dot{q}^i}; \dot{q}^i, 
\dot{p}_i= 
\frac{\partial \tilde{L}}{\partial q^i}
+\lambda_{\alpha}^1\frac{\partial\phi_1^{\alpha}}{\partial q^i}+ \lambda^2_{\alpha}\frac{\partial\phi_2^{\alpha}}{\partial q^i}\right)
\end{equation*}
satisfying  $\phi_1^{\alpha}(q,\dot{q})=0$ and $\phi_2^{\alpha}(q, \dot{q})=0$. 
\begin{example}\label{example3a}
If we use the same discretization map as in Example (\ref{example3}), 
\[
R^h_d(q, v)=
\left(q-\frac{h}{2}v, q+\frac{h}{2}v\right),
\]
then the constraints defining ${\mathcal T}_h^{R_d}N$ in Equation~\eqref{eq:TN_Rdv2} are: 
\[
\phi^{\alpha}_1(q, v)=\phi^{\alpha}\left(q-\frac{h}{2}v\right),\qquad \phi^{\alpha}_2(q, v)=\phi^{\alpha}\left(q+\frac{h}{2}v\right)\; .
\]
In the particular case of a mechanical system evolving on the sphere $N=S^2\subset {\mathbb R}^3$ with holonomic constraint $\phi(x,y,z)=x^2+y^2+z^2-1=0$, we have 
\begin{align*}
\phi_1(x,y,z, \dot{x}, \dot{y}, \dot{z})&=\left(x-\frac{h\dot{x}}{2}\right)^2+\left(y-\frac{h\dot{y}}{2}\right)^2+\left(z-\frac{h\dot{z}}{2}\right)^2-1,\\ \phi_2(x,y,z, \dot{x}, \dot{y}, \dot{z})&=\left(x+\frac{h\dot{x}}{2}\right)^2+\left(y+\frac{h\dot{y}}{2}\right)^2+\left(z+\frac{h\dot{z}}{2}\right)^2-1\; .
\end{align*}
These constraints determine the following subset
of $T{\mathbb R}^3=TQ$: 
\[
x^2+y^2+z^2+\frac{h^2}{4}(\dot{x}^2+\dot{y}^2+\dot{z}^2)=1, \qquad
x\dot{x}+y\dot{y}+z\dot{z}=0\; ,
\]
different from $TS^2$ described by $\{x^2+y^2+z^2=1\, , x\dot{x}+y\dot{y}+z\dot{z}=0\}$.
Note that the tangency condition on the right-hand side remains the same as in the sphere, but the holonomic constraint on the left-hand side is different in the modified contrained submanifold.
\end{example}
\subsection{Construction of constrained symplectic integrators}

 We use a similar philosophy as in the paper~\citep{21MBLDMdD} to obtain symplectic integrators for constrained systems defined by holonomic constraints. 
 
Consider the cotangent lift $R_d^{T^*}: TT^*Q\rightarrow T^*Q\times T^*Q$ of the discretization map $R_d\colon TQ \rightarrow Q\times Q$ (see Section~\ref{Sec:Rd}).

As $R_d^{T^*}$ is a symplectomorphism,  $R_d^{T^*}(\alpha_Q^{-1}(\Sigma^d_{\tilde{L}}))=R_d^{T^*}(\tilde{S}_d)$ is a Lagrangian submanifold of $(T^*Q\times T^*Q, \hbox{pr}_2^*(\omega_Q)-\hbox{pr}_1^*(\omega_Q))$. As summarized in Diagram~\eqref{Diagram:cotangentlift},
by construction we have that 
\begin{equation}\label{new-formula}
R_d^{T^*}(\tilde{S}_d)=\left(\Phi^{-1}\circ \widehat{R_d}\right)(\Sigma_{\tilde{L}}^d)\, .
\end{equation}

\begin{lemma} \label{Lemma:Sd}
If $(q_0, p_0; q_1, p_1)\in R_d^{T^*}(\tilde{S}_d) $, then 
\[\left(q_0, p_0+\eta^1_{\alpha}\,  {\rm d}\phi^{\alpha}(q_0); q_1, p_1+\eta^2_{\alpha}\,  {\rm d}\phi^{\alpha}(q_1)\right)\in R_d^{T^*}(\tilde{S}_d)\, .\]
\end{lemma}
\begin{proof}
The linearity in the  momenta of the applications involved in Equation~\eqref{new-formula}, together with the definition of ${\mathcal T}_h^{R_d}N$, guarantees the result. Remember that  $\phi^\alpha(q_0)=\phi^{\alpha}_1(q,v)$ and $\phi^\alpha(q_1)=\phi^{\alpha}_2(q,v)$, where $R_d(q,v)=(q_0,q_1)$.
\end{proof}

Consequently, the momenta in the Lagrangian submanifold associated with the discrete version of holonomic mechanical systems embedded into $T^*Q\times T^*Q$ are not uniquely determined. They depend on the choice of $\eta_{\alpha}^1$ and $\eta_{\alpha}^2$ in Lemma \ref{Lemma:Sd}.  

Similar to the continuous case described in Section~\ref{Sec:extr-intr} for the intrinsic description, the Lagrangian submanifold  $R_d^{T^*}(\tilde{S}_d)$ of $T^*Q\times T^*Q$ projects onto $N\times N$. Moreover, it can be projected to a  Lagrangian submanifold of $T^*N\times T^*N$ by the map $(i^*_N\times i^*_N)$.

\begin{proposition} \label{Prop:SNd}
The Lagrangian submanifold $R_d^{T^*}(\tilde{S}_d)$ of $(T^*Q\times T^*Q, \hbox{\rm pr}_2^*(\omega_Q)-\hbox{\rm pr}_1^*(\omega_Q))$ defines the Lagrangian submanifold
$ S_N^d$ of $(T^*N\times T^*N, \hbox{\rm pr}_2^*(\omega_N)-\hbox{\rm pr}_1^*(\omega_N))$ given by 
\begin{equation*} \label{eq:SNd}
S_N^d=(i^*_N\times i^*_N)\left(R_d^{T^*}(\tilde{S}_d)\right)\,.
\end{equation*}
\end{proposition}
\begin{proof}
The commutativity of the following diagram

\centerline{\xymatrix{  TT^*Q \ar[d]^{{\rm T}\pi_Q}  \ar[rr]^{R_d^{T^*}} && T^*Q\times T^*Q\ar[d]^{\hbox{pr}_1\times \hbox{pr}_2} \\TQ \ar[rr]^{R_d} && Q\times Q
}}
\noindent and the construction of ${\mathcal T}_h^{R_d}N$ in Equation~\eqref{eq:TN_Rdv2} guarantee that $R_d^{T^*}(\tilde{S}_d)$ projects onto $i_N(N)\times i_N(N)\equiv N\times N$ by $\hbox{pr}_1\times \hbox{ pr}_2$. In other words, $R_d^{T^*}(\tilde{S}_d) \subseteq T_N^*Q\times T_N^*Q$, where $T_N^*Q=\pi_Q^{-1}(N)$. 

As $i_{N}: N\hookrightarrow Q$ is the inclusion map, the map $i_N^*: T_N^*Q\rightarrow T^*N$ is well-defined and given by
\[
\langle i_N^*(\alpha_q), v_q\rangle
=\langle \alpha_q, {\rm T}i_N(v_q)\rangle\; ,
\]
for all $\alpha_q\in T_N^*Q$ and $v_q\in T_qN$. Hence, we can consider the subset 
\[
S_N^d=(i^*_{N}\times i^*_{N})\left(R_d^{T^*}(\tilde{S}_d)\right)
\]
of the symplectic manifold $(T^*N\times T^*N, \hbox{pr}_2^*(\omega_N)-\hbox{pr}_1^*(\omega_N))$.

Note that the maps
 $i_N\times i_N$ and $({\rm pr}_1\times {\rm pr}_2)|_{R_d^{T^*}(\tilde{S}_d)}$ are transverse\footnote{We recall that if $M, N, P$ are differentiable manifolds and $f: M\rightarrow P$ and $g: N\rightarrow P$ are two differentiable maps, we say that $f$ and $g$ are transversal if
 $T_xf(T_xM)+T_yg(T_yN)=T_zP$ for all $x\in M$, $y\in N$ verifying $z=f(x)=g(y)$.}. By \cite[Section 4.3]{13GuiStern}, we deduce that $S^d_N$
is a Lagrangian submanifold of $(T^*N\times T^*N, \hbox{pr}_2^*(\omega_N)-\hbox{pr}_1^*(\omega_N))$.
\end{proof}

Lemma~\ref{Lemma:Sd} shows that the correspondence between the two Lagrangian submanifolds in Proposition~\ref{Prop:SNd} is not one-to-one, but there are infinitely many momenta in $T^*Q\times T^*Q$ that correspond with the same one in $S^d_N$.

The above discussion makes it  possible to construct symplectic integrators for holonomic mechanical systems as stated in the following result.

\begin{theorem}\label{Theorem:main} Let $L: TN\rightarrow {\mathbb R}$ be a regular Lagrangian function where $N$ is a submanifold of $Q$ and $R_d \colon TQ \rightarrow Q\times Q$ be a discretization map. Let $\tilde{L}: TQ\rightarrow {\mathbb R}$ be an arbitrary regular extension of $L$ to $TQ$. 
Let $\tau_{T^*Q}\colon TT^*Q\rightarrow T^*Q$ be the canonical tangent projection, then the following conditions
\begin{align} 
\dfrac{1}{h}\, \left(R_d^{T^*}\right)^{-1}(q_0, p_0; q_1, p_1)&\in   \tilde{S}_d\left(\left(\tau_{T^*Q}\circ \left(\dfrac{1}{h}\, \left(R_d^{T^*}\right)^{-1}\right)\right)(q_0, p_0; q_1, p_1)\right), \label{Eq:HMethod} \\
(q_0,p_0;q_1,p_1)&\in Leg_{\tilde{L}}(TN) \times Leg_{\tilde{L}}(TN),\label{Eq:HMethod2}
\end{align}
define a constrained symplectic integrator for $L$ on $T^*_NQ$, where $Leg_{\tilde{L}}: TQ\rightarrow T^*Q$ is the Legendre transformation associated to $\tilde{L}: TQ\rightarrow {\mathbb R}$.

Alternatively, the set of discrete equations given by 
\begin{eqnarray} 
(q_0, p^N_0; q_1, p^N_1)&\in&   (i^*_{N}\times i^*_{N})\left( R_d^{T^*}\left(h\, \tilde{S}_d\right)\right)\label{Eq:HMethod-N} 
\end{eqnarray}
defines a symplectic integrator for $L$ where $(q_i, p_i^N)\in T^*N$, $i=0, 1$.(Here, the notation $h\tilde{S}_d$ means fiber multiplication by $h$ with respect to the vector bundle structure $\tau_{T^*Q}: TT^*Q\rightarrow T^*Q$ of the elements of $\tilde{S}_d$.)
\end{theorem}
\begin{proof}
Equations~\eqref{Eq:HMethod} and~\eqref{Eq:HMethod-N} define  symplectic methods because $\tilde{S}_d$ is a Lagrangian submanifold that guarantees the symplecticity (see Proposition \ref{Prop:SNd}). Equation~\eqref{Eq:HMethod2} uniquely determines the value of the Lagrange multipliers. 
    
   Moreover,  Proposition~\ref{Prop:SNd} guarantees that from $\tilde{S}_d$ it is possible to construct the Lagrangian submanifold $ S_N^d$ of $(T^*N\times T^*N, \hbox{pr}_2^*(\omega_N)-\hbox{pr}_1^*(\omega_N))$. Thus, $(i_N^*\times i_N^*) (q_0,p_0;q_1,p_1)$ lives in $ T^*N \times T^*N$ and defines a symplectic integrator for $L$ under the assumption of regularity.
\end{proof}

Alternatively, we can write Equations (\ref{Eq:HMethod}) and (\ref{Eq:HMethod2}) as 
\begin{equation*} 
(q_0, p_0; q_1, p_1)\; \in \;  R_d^{T^*}({h}\tilde{S}_d)\cap \left(Leg_{\tilde{L}}(TN) \times Leg_{\tilde{L}}(TN)\right)\, .\label{Eq:HMethod-alte-1} 
\end{equation*} 
As stated in Lemma~\ref{Lemma:Sd}, the condition $(q_0, q_1)\in N\times N$ and Equation~\eqref{Eq:HMethod}  do not determine unique solutions for momenta in $T^*_NQ \times T^*_NQ$.
However, the additional Equation~\eqref{Eq:HMethod2} does.  Thus,  Equations~\eqref{Eq:HMethod} and~\eqref{Eq:HMethod2} define a symplectic integrator on $T^*N\times T^*N$ as shown in Equation~\eqref{Eq:HMethod-N}.

\subsection{A geometrical interpretation of the discrete null space method}
Let $\{X_a\}$ be a basis of vector fields on $N$ that can be embedded into $T_qQ$ by the tangent map ${\rm T}i_N\colon T_qN\rightarrow T_{{\rm i}_N(q)}Q$. For any element $p^Q$  in $T_N^*Q$, a momenta $p^N$ in $T^*N$
is uniquely determined by the following projection:
\begin{equation}\label{eq:pN_pQ}
\left\langle p^N, X_a(q)\right\rangle_{T^*N}=\left\langle p^Q, {\rm T}i_N\left(X_a(q)\right)\right\rangle_{T^*Q}\, .
\end{equation}
These $n$-equations univocally determine $p^N$ from $p^Q$.
 Equation (\ref{Eq:HMethod-alte-1}) can be rewritten as: 
\begin{eqnarray} 
\hspace{-0.5cm} p^Q_{0} &\in & {\rm pr}_1\left(R_d^{T^*}({h}\tilde{S}_d)\right)\cap (\pi_Q)^{-1}(q_0), \label{Eq:null-1}\\
\hspace{-0.5cm} p^Q_{1} &\in & {\rm pr}_2\left(R_d^{T^*}({h}\tilde{S}_d)\right)\cap (\pi_Q)^{-1}(q_1), \label{Eq:null-2}\\
\left(p^Q_{0},p^Q_{1}\right)&\in& Leg_{\tilde{L}}(TN) \times Leg_{\tilde{L}}(TN)\, .\nonumber
\end{eqnarray}

Equations~\eqref{eq:pN_pQ}, \eqref{Eq:null-1} and \eqref{Eq:null-2} give a geometrical interpretation of the  null space condition in~\cite{Betsch, BetschII}  to obtain numerical integrators for holonomic mechanical systems without Lagrange multipliers (see Section \ref{Sec:ExampleEuclid}). 

 The following diagram illustrates the construction process:

\centerline{\xymatrix{TT^*Q \ar[d]_{{\rm T}\pi_Q}   && {h}\tilde{S}_d \ar@{_{(}->}[ll] \ar[d]_{{\rm T}\pi_Q|_{{h}\tilde{S}_d}}  \ar[rr]^(.4){R_d^{T^*}|_{{h}\tilde{S}_d}} && T_N^*Q\times T_N^*Q\ar[d]^{{\rm pr}_1\times {\rm pr}_2} \ar[rr]^{i_N^*\times i_N^*} && T^*N \times T^*N \ar[dll]^{{\rm pr}_1\times {\rm pr}_2}\\ TQ && {\mathcal T}_h^{R_d}N\ar@{_{(}->}[ll]\ar[rr]^{R^h_d} && N\times N
}}
\vspace{2mm}

Note that $R_d^{T^*}({h}{\tilde{S}_d})$ is a $2m$-dimensional submanifold of $T^*_NQ\times T^*_NQ$ by definition {with $m=\dim Q$}. Moreover, {in general}, $\tilde{S}_d$ does not project onto $T_N^*Q$ by the canonical projection $\tau_{T^*Q}\colon TT^*Q \rightarrow T^*Q$ because the base point does not necessarily live in $N$, but in $Q$. As described in Theorem~\ref{Theorem:main} we have imposed that the holonomic constraints are exactly satisfied by the discrete flow, not necessarily in the continuous counterpart.   
Indeed, the corresponding point $\left(R^h_d\right)^{-1}(q_0,q_1)$ in the continuous setting lives in ${\mathcal T}_h^{R_d}N$,  defined in Equation~\eqref{eq:TN_Rdv2}, that usually is not equal to $TN$, understood as a submanifold of $TQ$.

In the following section, we will show that our construction provides a new interpretation of well-known methods for constrained mechanics, as the ones described in~\cite{LeRe}. 

\section[Application for mechanical Lagrangian systems]{Application to
mechanical Lagrangian systems with holonomic constraints} \label{Sec:ExampleEuclid}

Assume that $\tilde{L}\colon TQ \rightarrow \mathbb{R}$ is a regular Lagrangian such that $\tilde{L}|_{TN}=L\colon TN \rightarrow \mathbb{R}$, that is, it is an extension of the Lagrangian $L$ to $TQ$. To extrinsically obtain a numerical integrator for the corresponding holonomic Euler-Lagrange equations, we follow the construction in Theorem \ref{Theorem:main}. For every starting point $(q_0,p_0)\in {\rm Leg}_{\tilde{L}}(TN)$ there exists a unique $(q_1,p_1)\in {\rm Leg}_{\tilde{L}}(TN)$ satisfying Equations~\eqref{Eq:HMethod} and~\eqref{Eq:HMethod2}.
The map $i_N^*\colon T^*_NQ\rightarrow T^*N$ associated to  the natural inclusion $i_N\colon N\hookrightarrow Q$ can be used to (non-canonically) identify  ${\rm Leg}_{\tilde{L}}(TN)$ to $T^*N$.

The following diagram shows how to move between the different spaces.

\centerline{\xymatrix{  && T_N^*Q\times T_N^*Q  \ar[rr]^{i_N^*\times i_N^*} && T^*N \times T^* N  \ar[d]^{{\rm Leg}_L^{-1}\times {\rm Leg}_L^{-1}}   \\ TT^*Q  \supset{h}\tilde{S}_d \ar[d]_{{\rm T}\pi_Q}  \ar[rru]^{R_d^{T^*}|_{{h}\tilde{S}_d}}&& {\rm Leg}_{\tilde{L}}(TN)\times{\rm Leg}_{\tilde{L}}(TN)  \ar[urr]^{\left(i_N^*\times i_N^*\right)|_{{\rm Leg}_{\tilde{L}}(TN)  }} \ar@{^(->}[u] \ar[d]^{\pi_Q\times \pi_Q}  && TN\times TN  \ar[ll]_(0.4){{\rm Leg}_{\tilde{L}}\times {\rm Leg}_{\tilde{L}} }\ar[dll]^{\tau_N\times \tau_N}\\{\mathcal T}_h^{R_d}N\ar[rr]^{{R^h_d}} && N\times N &&
}}
\vspace{2mm}
Let us rewrite Equation~\eqref{eq:pN_pQ} when local coordinates adapted to $N$ are used, that is, $N=\{q=(q^a,q^\alpha)\in Q \, | \, \phi^\alpha(q)=q^\alpha-\Phi^\alpha(q^a)=0\}$. As   
$$i_N^*(q^a,p^Q_i)=\left(q^a,p^Q_a+p^Q_\alpha \dfrac{\partial \Phi^\alpha}{\partial q^a} \right)=(q^a,p_a^N)\, ,$$
Equation~\eqref{eq:pN_pQ} becomes
$$\langle p^Q, {\rm T}i_N\left(X_a(q)\right)\rangle_{T^*Q}=\langle i^*_N(p), X_a(q)\rangle_{T^*N}= \left\langle p^Q_a+p^Q_\alpha \, \dfrac{\partial \Phi^\alpha}{\partial q^a}, X_a(q)\right\rangle_{T^*N},$$
where the subscripts emphasize the manifold where the natural pairings are computed.

\subsection{On Euclidean vector spaces}
In this section, the manifold $Q$ is taken to be ${\mathbb R}^m$ and consider the mechanical Lagrangian $\tilde{L}: T{\mathbb R}^m\rightarrow {\mathbb R}$ given by
\[
\tilde{L}(q, v)=\frac{1}{2}v^T{\mathbf M}v-V(q)\, ,
\]
where ${\mathbf M}$ represents  a positive definite and symmetric mass matrix and $V$ is a potential function on $Q$. We also assume that the system is subjected to   the holonomic constraints $\phi^\alpha(q)=0$, $1\leq \alpha\leq m-n$, that define the submanifold $N$ of $Q$.
\subsubsection{Symplectic Euler method-A}\label{section:euler-a}

From the retraction map $R(q,v)=q+v$, we define the following discretization map  
\[
R_d (q, v)=(q, q+v) \, ,
\]
whose inverse map is $$(R_d)^{-1}(q_0, q_1)=\left(q_0, q_1-q_0\right).$$
The modified submanifold  ${\mathcal T}_h^{R_d}N=(R^h_d)^{-1}(N\times N)$ is given by the points $(q,v)$ in $TQ$ such that
\[
\phi^\alpha_1(q, v)=\phi^{\alpha}(q)=0\ \hbox{  and }\  \phi^\alpha_2(q, v)=\phi^\alpha (q+hv)=0\, .
\]
Moreover, 
\begin{eqnarray*}
R_d^{T^*}(q, p, \dot{q}, \dot{p})&=&\left(q,p-\dot{p}; q+\dot{q}, p \right)  \,,  \\
\dfrac{1}{h}\, \left(R_d^{T^*}\right)^{-1}(q_0, p_0; q_1, p_1)
&=&\left(q_0,p_1; \dfrac{q_1-q_0}{h}, \dfrac{p_1-p_0}{h}\right)   \,,\end{eqnarray*}
because the factor $1/h$ only multiplies the fibers of $TT^*Q$ with respect to $\tau_{T^*Q}$.
After some computations, Equations~\eqref{Eq:HMethod} and~\eqref{Eq:HMethod2} become:
\begin{eqnarray*}
p_1&=&{\mathbf M}\frac{q_1-q_0}{h}+h\tilde{\lambda}_\alpha \nabla \phi^{\alpha}  (q_1)\, ,\\
\frac{p_1-p_0}{h}&=&-\nabla V(q_0)+\lambda_\alpha \nabla \phi^\alpha (q_0)+\tilde{\lambda}_\alpha\nabla  \phi^\alpha  (q_1) \, ,\\
(q_0,p_0)&\in& {\rm Leg}_{\tilde{L}}(TN),  \qquad  (q_1,p_1)\in {\rm Leg}_{\tilde{L}}(TN).
\end{eqnarray*}
Consequently, these equations rewritten as in Equations~\eqref{Eq:HMethod-N} 
define a symplectic method on $T^*N$.
Alternatively, we can write the above equations as follows 
\begin{eqnarray*}
p_0&=&{\mathbf M}\left(\frac{q_1-q_0}{h}\right)+h\nabla V(q_0)-h\lambda_\alpha \nabla \phi^\alpha (q_0)\, ,\\
p_1&=&{\mathbf M}\left(\frac{q_1-q_0}{h}\right)+h\tilde{\lambda}_\alpha \nabla \phi^{\alpha}  (q_1)\, ,\\
{(q_0, {\mathbf M}^{-1}p_0)}&\in&TN,  \qquad  {(q_1,{\mathbf M}^{-1}p_1)\in TN.}
\end{eqnarray*}

\subsubsection{Symplectic Euler method-B}
\label{section:euler-b}

From the retraction map $R(q,v)=q-v$, we construct the following discretization map  
\[
R_d (q, v)=(q-v, q) \, ,
\]
whose inverse map is  $$(R_d)^{-1}(q_0, q_1)=\left(q_1, q_1-q_0\right).$$
Then ${\mathcal T}_h^{R_d}N=(R^h_d)^{-1}(N\times N)$ is given by the points $(q,v)$ in $TQ$ such that
\[
\phi^\alpha_1(q, v)=\phi^\alpha(q-hv)=0 \hbox{  and }  \phi^\alpha_2(q, v)=\phi^\alpha (q)=0\, .
\]
Moreover, 
\begin{eqnarray*}
R_d^{T^*}(q, p, \dot{q}, \dot{p})&=&\left(q-\dot{q},p; q, \dot{p}+p \right)  \,,  \\
\dfrac{1}{h} \left(R_d^{T^*}\right)^{-1}(q_0, p_0; q_1, p_1)
&=&\left(q_1,p_0; \dfrac{q_1-q_0}{h}, \dfrac{p_1-p_0}{h}\right)   \,.\\ \,.\end{eqnarray*}
After some computations,  Equations~\eqref{Eq:HMethod} and~\eqref{Eq:HMethod2} lead to the following constrained symplectic method 
\begin{eqnarray}
p_0&=&{\mathbf M}\frac{q_1-q_0}{h}-h{\lambda}_\alpha \nabla \phi^{\alpha}  (q_0)\, ,  \label{Eq:EulerB1}\\
\frac{p_1-p_0}{h}&=&-\nabla V(q_1)+\lambda_\alpha \nabla \phi^\alpha (q_0)+\tilde{\lambda}_\alpha\nabla  \phi^\alpha  (q_1)\, , \label{Eq:EulerB2}\\
(q_0,p_0)&\in& {\rm Leg}_{\tilde{L}}(TN),  \qquad  (q_1,p_1)\in {\rm Leg}_{\tilde{L}}(TN). \nonumber
\end{eqnarray}

Substituting Equation~\eqref{Eq:EulerB1} for the tuples $(q_0,p_0;q_1,p_1)$ and $(q_1,p_1;q_2,p_2)$  in Equation~\eqref{Eq:EulerB2}
we obtain the next point $q_2$ in the discrete flow on $N$ with the following numerical method:
\begin{eqnarray*}
{\mathbf M}\frac{q_2-2q_1+q_0}{h^2}&=&-\nabla V(q_1)+\mu_{\alpha}\nabla  \phi^\alpha  (q_1)\, ,\\
\phi^\alpha(q_2)&=& 0\, ,
\end{eqnarray*}
where $\mu_{\alpha}=\tilde{\lambda}^{(0)}_{\alpha}+\lambda^{(1)}_{\alpha}$. The Lagrange multipliers with the superscript $(0)$ and $(1)$ are coming from the Equation~\eqref{Eq:EulerB2} for the points $(q_0,p_0;q_1,p_1)$ and $(q_1,p_1;q_2,p_2)$, respectively. That method is precisely the SHAKE discretization in~\citep{LeimSkeel96} and~\cite[Chapter 7.2]{LeRe} {which, being formulated as a $2$-step method on $N$, sidesteps the necessity of considering the restriction on the momenta. Nevertheless, in order to include a velocity or momentum level representation, the tangency condition needs to be reintroduced. There is, however, no unique way to move back to momentum level.}

The symplectic method in Equations~\eqref{Eq:EulerB1} and~\eqref{Eq:EulerB2} can be rewritten as follows:   
\begin{eqnarray}
p_0&=&{\mathbf M}\left(\frac{q_1-q_0}{h}\right)-h\lambda_\alpha \nabla \phi^\alpha (q_0) \, , \label{eq1} \\
p_1&=&{\mathbf M}\left(\frac{q_1-q_0}{h}\right)-h\nabla V(q_1)+h\tilde{\lambda}_\alpha \nabla \phi^{\alpha}  (q_1)\, , \label{eq2}\\
\phi^\alpha(q_0)&=&0\, , \qquad  \nabla\phi^{\alpha}(q_0){\mathbf M}^{-1}p_0=0\, , \label{eq:momentap0} \\
\phi^\alpha(q_1)&=&0\, ,\qquad\nabla\phi^{\alpha}(q_1){\mathbf M}^{-1}p_1=0\, . \label{eq:momentap1}
\end{eqnarray}
As the Lagrangian is regular, the discrete flow on ${\rm Leg}_{\tilde{L}}(TN)$ determines a discrete flow on $TN$ that satisfies both the constraint equation and the tangency condition.
{This is a modification of} the RATTLE algorithm in~\cite{LeimSkeel96} and~\cite[Chapter 7.2]{LeRe}. {It is in fact an order 1 version of it. If RATTLE is the order 2 version of SHAKE at momentum level, Equations~\eqref{eq1}-~\eqref{eq:momentap1} can be regarded as an order 1 extension and can be thought of as constraining the Hamiltonian map of the discrete Lagrangian \cite[cf. Section 3.5.4]{MW_Acta}
\begin{equation*}
L_d(q_0,q_1,h) = h L\left(q_1, \frac{q_1 - q_0}{h}\right).
\end{equation*}
}

{\bf Initial conditions discussion}.
Assume that we have the following initial conditions $(q_0, p_0)$ in $T^*\mathbb{R}^m$ such that 
$\phi^{\alpha} (q_0)=0$ and $\nabla\phi^{\alpha}(q_0){\mathbf M}^{-1}p_0=0$, that is, $(q_0, p_0)\in Leg_{\tilde L}(TN)$, equivalently $(q_0,{\mathbf M}^{-1}p_0)\in TN$.  {{Lagrange multipliers in Equations~\eqref{eq1} and~\eqref{eq2} are chosen so that the constraint and tangency conditions are both satisfied.}} By replºacing Equation~\eqref{eq1} in the right-hand side equation in~\eqref{eq:momentap0} coming from $(q_0, p_0)\in Leg_{\tilde L}(TN)$, 
we determine the Lagrange multiplier ${\lambda}_\beta$ as follows
\[
0=\nabla\phi^{\alpha}(q_0)\frac{q_1-q_0}{h}-h{\lambda}_\beta \nabla \phi^{\alpha} (q_0){\mathbf M}^{-1}\nabla \phi^{\beta}  (q_0)\, .
\]
If we denote by $({\mathcal C}_{\alpha\beta})$ the inverse matrix of 
\[
({\mathcal C}^{\alpha\beta})=(\nabla \phi^{\alpha} (q_0){\mathbf M}^{-1}\nabla \phi^{\beta}(q_0)),
\]
then 
$
{\lambda}_\beta= {\mathcal C}_{\alpha\beta} \nabla\phi^{\alpha}(q_0)\frac{q_1-q_0}{h^2}
$.
Therefore, Equation~\eqref{eq1} becomes
\[
p_0={\mathbf M}\frac{q_1-q_0}{h}-  \left({\mathcal C}_{\beta\alpha} \nabla\phi^{\beta}(q_0)\frac{q_1-q_0}{h}\right)\nabla\phi^{\alpha}  (q_0)\, .
\]
Using this equation and the holonomic constraint we find the value of $q_1$. 

Alternatively, Equation~\eqref{eq:pN_pQ} leads to
the $m$-implicit equations to solve for $q_1$:
\[
\langle p^N_0, X_a(q_0)\rangle_{T^*N}=\left\langle {\mathbf M}\frac{q_1-q_0}{h},{\rm T}i_N\left(X_a(q_0)\right)\right\rangle_{T^*Q},  \quad \phi^{\alpha}(q_1)=0\, .
\]

Finally, we determine the Lagrange multiplier $\tilde \lambda_{\beta}$ using the value of $p_1$ from~\eqref{eq2} in $(q_1,{\mathbf M}^{-1}p_1)\in TN$, that is, $\nabla\phi^{\alpha}(q_1){\mathbf M}^{-1}p_1=0$.

Thus, we implicitly determine a map $(q_0, p_0)\rightarrow (q_1, p_1)$ such that $(q_0, p_0), (q_1, p_1)\in Leg_{\tilde L}(TN)$ because for an initial condition in $Leg_{\tilde L}(TN)$, there exist Lagrange multipliers so that $(q_1, p_1)\in Leg_{\tilde L}(TN)$.

\subsubsection{Mid-point discretization}

Let us consider the discretization map 
\[
R_d (q, v)=\left(q-\frac{1}{2}v, q+\frac{1}{2}v\right) 
\]
that is associated with the mid-point rule because the inverse map is $$(R_d)^{-1}(q_0, q_1)=\left(\frac{q_0+q_1}{2}, q_1-q_0\right).$$

Therefore, the modified submanifold  ${\mathcal T}_h^{R_d}N=(R^h_d)^{-1}(N\times N)$ of $TQ$  is given by 
\[
\phi^\alpha_1(q, v)=\phi^{\alpha}\left(q-\frac{h}{2}v\right)=0 \hbox{  and }  {\phi}^\alpha_2(q, v)=\phi^{\alpha}\left(q+\frac{h}{2}v\right)=0 \, .
\]
Moreover, 
\begin{eqnarray*}
R_d^{T^*}(q, p, \dot{q}, \dot{p})&=&\left(q-\dfrac{1}{2}\,\dot{q},p-\dfrac{\dot{p}}{2}; q+\dfrac{1}{2}\,\dot{q},p+\dfrac{\dot{p}}{2}\right)  \,,  \\
\dfrac{1}{h}\, \left(R_d^{T^*}\right)^{-1}(q_0, p_0; q_1, p_1)
&=&\left(\frac{q_0+q_1}{2},\dfrac{p_0+p_1}{2} \,; \dfrac{q_1-q_0}{h}, \dfrac{p_1-p_0}{h}\right)  \,.\\ \,.\end{eqnarray*}

After some calculations Equations~\eqref{Eq:HMethod} and~\eqref{Eq:HMethod2} define by construction the following symplectic method:
\begin{eqnarray*}
\frac{p_0+p_1}{2}&=&{\mathbf M}\frac{q_1-q_0}{h}-\frac{h}{2}{\lambda}_\alpha \nabla \phi^{\alpha}  (q_0)+\frac{h}{2}\tilde{\lambda}_\alpha \nabla \phi^{\alpha}  (q_1)\, ,\\
\frac{p_1-p_0}{h}&=&-\nabla V\left(\frac{q_0+q_1}{2}\right)+{\lambda}_\alpha \nabla \phi^{\alpha}  (q_0)+\tilde{\lambda}_\alpha \nabla \phi^{\alpha}  (q_1)\, ,\\
(q_0,p_0)&\in& {\rm Leg}_{\tilde{L}}(TN),  \qquad  (q_1,p_1)\in {\rm Leg}_{\tilde{L}}(TN).
\end{eqnarray*}
Alternatively, we can write the equations as follows 
\begin{eqnarray*}
p_0&=&{\mathbf M}\left(\frac{q_1-q_0}{h}\right)+\frac{h}{2}\nabla V\left(\frac{q_0+q_1}{2}\right)-h\lambda_\alpha \nabla \phi^\alpha (q_0) \, ,\\
p_1&=&{\mathbf M}\left(\frac{q_1-q_0}{h}\right)-\frac{h}{2}\nabla V\left(\frac{q_0+q_1}{2}\right)+h\tilde{\lambda}_\alpha \nabla \phi^\alpha (q_1)\, ,\\
(q_0,p_0)&\in& {\rm Leg}_{\tilde{L}}(TN),  \qquad  (q_1,p_1)\in {\rm Leg}_{\tilde{L}}(TN).
\end{eqnarray*}

\subsection{Composition of symplectic Euler methods: RATTLE discretization}\label{Sec:Composition}
 The composition of different methods allows to derive higher-order methods \citep{hairer}. Such a composition can also be interpreted as the composition of Lagrangian submanifolds. 

Let $(P_i, \omega_i)$ be symplectic manifolds $i=1, 2, 3$. 
Denote by $P^-_i\times P_j$ the symplectic manifold with  the symplectic 2-form $\hbox{pr}_j^*\omega_j-\hbox{pr}_i^*\omega_i$ where $\hbox{pr}_i: P_i\times P_j\rightarrow P_i$ and $\hbox{pr}_j: P_i\times P_j\rightarrow P_j$  are the canonical projections. 

If $\Sigma_1\subset P^-_1\times P_2$ and $\Sigma_2\subset P^-_2\times P_3$ are Lagrangian submanifolds, then the composition
\[
\Sigma_2\circ \Sigma_1\subset P_1^-\times P_3
\]
is a Lagrangian submanifold (under clean intersection conditions, see \cite{13GuiStern}) where
\[
\Sigma_2\circ \Sigma_1=\{ (x_1, x_3)\in P_1\times P_3\; |\; \exists\;  x_2\in P_2 \mbox{  such that } (x_1, x_2)\in \Sigma_1, (x_2, x_3)\in \Sigma_2\}\, .
\]
Consider a holonomic system given by a Lagrangian $L: TN\rightarrow {\mathbb R}$ where $N$ is an embedded  submanifold of the ambient manifold $Q$. From a discretization map $R_d: TQ\rightarrow Q\times Q$ we obtain the Lagrangian submanifold ${\mathcal L}^h=R_d^{T^*}(h\, \tilde{S}_d)$ of $T^*Q\times T^*Q$.  For real numbers $\gamma_1, \ldots, \gamma_s$, we define the Lagrangian submanifold
\begin{equation}\label{comp}
{\mathcal L}^{\gamma_sh}\circ \ldots \circ {\mathcal L}^{\gamma_1h}
\end{equation}
that generates a symplectic composition method:
{\begin{equation*} 
(q_0, p_0; q_1, p_1)\; \in \; \left({\mathcal L}^{\gamma_sh}\circ \ldots \circ {\mathcal L}^{\gamma_1h}\right) \cap \left(Leg_{\tilde{L}}(TN) \times Leg_{\tilde{L}}(TN)\right)\, .
\end{equation*}  
}
Another non equivalent option is obtained by composing the Lagrangian submanifold 
${\mathcal L}_N^h=(i^*_N\times i^*_N)\left( R_d^{T^*}(h\, \tilde{S}_d)\right)\,$ of $(T^*N\times T^*N,{\rm pr}_2^*\omega_N- {\rm pr}_1^*\omega_N)$ as follows:
\begin{equation}\label{comp2}
{\mathcal L}_N^{\gamma_sh}\circ \ldots \circ {\mathcal L}_N^{\gamma_1h}.
\end{equation} The resulting symplectic integrator is different from the one in Equation~\eqref{comp} because the composition of Lagrangian submanifolds has taken place on $T^*N\times T^*N$, in contrast with the composition in Equation~\eqref{comp} that has taken place extrinsically on $T^*Q\times T^*Q$.

If we start from a method of even order $p$ and the coefficients $\gamma_1, \ldots, \gamma_s$ verify 
\begin{align*}
\gamma_1+ \ldots + \gamma_s = 1\, ,\\
\gamma^{p+1}_1+ \ldots + \gamma^{p+1}_s = 0,
\end{align*}
then the symplectic composition {methods~\eqref{comp} and~\eqref{comp2} are at least of order $p+1$ in $TN$~\citep[Sections II.4 and IV.4]{hairer}}

Alternatively, we can also compose using a discretization map $R_d$ and  the corresponding adjoint map $R_d^*$. If 
${\mathcal L}^h=R_d^{T^*}(h\, \tilde{S}_d)$ and we denote the Lagrangian submanifold generated by the adjoint map by 
$({\mathcal L}^h)^*= (R_d^*)^{T^*}(h\,\tilde{S}_d)$, then the operation
\begin{equation}\label{comp-3}
\left({\mathcal L}^{\beta_sh}\right)^*\circ {\mathcal L}^{\alpha_sh}\circ \ldots \circ \left({\mathcal L}^{\beta_1h}\right)^*\circ {\mathcal L}^{\alpha_1h}
\end{equation}
generates a symplectic composition method for the constrained symplectic method (see~\cite{hairer} for conditions on the coefficients $\alpha_l$ and $\beta_l$, $1\leq l\leq s$).

As in~\eqref{comp2} we obtain a different class of symplectic integrators considering 
\begin{equation}\label{comp-4}
\left({\mathcal L}_N^{\beta_sh}\right)^*\circ {\mathcal L}_N^{\alpha_sh}\circ \ldots \circ \left({\mathcal L}_N^{\beta_1h}\right)^*\circ {\mathcal L}_N^{\alpha_1h}\, .
\end{equation}

\begin{example}\label{compos-example}
For instance, composing the  methods defined in Sections~\ref{section:euler-a} and~\ref{section:euler-b} we obtain the following  second order method (RATTLE discretization) using Equation~\eqref{comp-3} where ${\mathcal L}^{\alpha_1h}$ corresponds with the symplectic Euler method A, $({\mathcal L}^{\beta_1h})^*$ with the symplectic Euler method B for $\alpha_1=\beta_1=1/2$: 
{
\begin{eqnarray}
p_0&=&{\mathbf M}\frac{q_{1/2}-q_0}{h/2}+\frac{h}{2}\nabla V(q_0)-\frac{h}{2}{\lambda}_\alpha \nabla \phi^{\alpha}  (q_0)\, ,\label{eq:Composition1}\\
{p_{1/2}}&{=}&{{\mathbf M}\frac{q_{1/2}-q_0}{h/2}+\dfrac{h}{2}\tilde{\lambda}_{\alpha}\nabla \phi^{\alpha}  (q_{1/2})}\label{eq:Composition_cond1}\\
&{=}&{{\mathbf M}\frac{q_{1}-q_{1/2}}{h/2}-\dfrac{h}{2}{\Lambda}_{\alpha}\nabla \phi^{\alpha}  (q_{1/2})\,,}\label{eq:Composition_cond2}\\
p_1&=&{\mathbf M}\frac{q_1-q_{1/2}}{h/2}-\frac{h}{2}\nabla V(q_1)+\frac{h}{2}\tilde{\Lambda}_\alpha\nabla  \phi^\alpha  (q_1) \, , \label{eq:Composition2}\\
(q_0,p_0)&\in& {\rm Leg}_{\tilde{L}}(TN),  \qquad  (q_1,p_1)\in {\rm Leg}_{\tilde{L}}(TN)\, . \label{eq:Composition3}
\end{eqnarray}
Equations~\eqref{eq:Composition_cond1} and~\eqref{eq:Composition_cond2} conform the composability condition of the two Lagrangian submanifolds. This composability may be achieved in $T^*N$ requiring that
\begin{equation*}
(q_{1/2},p_{1/2}) \in {\rm Leg}_{\tilde{L}}(TN),
\end{equation*}
which corresponds to a composition of the type of \eqref{comp2}. However, we may simply require composability extrinsically, i.e. only in $T^*Q$, as in \eqref{comp}. For instance, we have that ${\Lambda}_{\alpha} = - \tilde{\lambda}_{\alpha}$, which results in $q_{1/2} = \frac{q_1 + q_0}{2}$. Thus, disregarding $p_{1/2}$ entirely, we obtain
\begin{eqnarray}
p_0&=&{\mathbf M}\frac{q_1-q_0}{h}+\frac{h}{2}\nabla V(q_0)-\frac{h}{2}{\lambda}_\alpha \nabla \phi^{\alpha}  (q_0)\, ,\label{eq:RattleCompose1}\\
p_1&=&{\mathbf M}\frac{q_1-q_{0}}{h}-\frac{h}{2}\nabla V(q_1)+\frac{h}{2}\bar{\lambda}_\alpha\nabla  \phi^\alpha  (q_1) \, , \label{eq:RattleCompose2}\\
(q_0,p_0)&\in& {\rm Leg}_{\tilde{L}}(TN),  \qquad  (q_1,p_1)\in {\rm Leg}_{\tilde{L}}(TN)\, . \label{eq:RattleCompose3}
\end{eqnarray}
These equations conform the well-known order 2 RATTLE method, \cite{RattleAndersen}, which can be understood as an application of a constrained Lobatto IIIAB method \citep{Jay1996,Jay2015}}.
\end{example}

\begin{remark}
The previous examples are well-known in the literature but we are giving a new point of view that can be easily generalized to more involved situations as we will show in the next example and generalize in the next section.  

On the other hand, these methods are closely related with the  projection technique that allows us to apply symplectic integrators, suitable for the integration of unconstrained systems, to the integration of constrained
systems as well \citep{LeimSkeel96,Reich1}.
\end{remark}

\begin{example}{\bf Rigid body as a constrained system.}\label{Ex:RigidBody}
Consider a Lagrangian $L: T(SO(3)\times {\mathbb R}^3)\rightarrow {\mathbb R}$ given by
\[
L(R, x, \dot{R}, \dot{x})=\frac{1}{2}\hbox{tr} (\dot{R}{J}\dot{R}^T)+\frac{1}{2}m||\dot{x}||^2-V(R, x)\, ,
\]
where $J$ is the symmetric {second moments of} mass tensor of the body, {related to the standard moment of inertia tensor of the body $I$ by $J = \frac{1}{2}\mathrm{tr}(I)\, \mathrm{Id}_3 - I$}. 
Extrinsically, this holonomic system can be considered on $T({\mathbb R}_{3\times 3}\times {\mathbb R}^3)$ introducing the orthogonality constraints
$R^T R={\rm Id}_3$: 
\[
\phi_i(R)=\frac{1}{2}\hbox{tr}\left[ \left((R^TR)-{\rm Id}_3\right) \Lambda_i\right]\, ,
\]
where $\Lambda_i$, $1\leq i\leq 6$, is a basis of symmetric matrices.

Using the composition method in Equations~\eqref{eq:RattleCompose1},~\eqref{eq:RattleCompose2} and~\eqref{eq:RattleCompose3} we derive the discrete equations: 
\begin{eqnarray*}
P_0&=&\frac{R_{1}-R_0}{h}{J}+\frac{h}{2}\nabla_1 V(R_0, x_0)-\frac{h}{2}R_0{\Lambda} \, ,\\
p_0&=& m\frac{x_{1}-x_0}{h}+\frac{h}{2}\nabla_2 V(R_0, x_0)\, ,\\
P_1&=&\frac{R_1-R_{0}}{h}{J}-\frac{h}{2}\nabla_1 V(R_1, x_1)+\frac{h}{2}R_1\tilde{\Lambda} \, ,\\
p_1&=& m\frac{x_{1}-x_0}{h}-\frac{h}{2}\nabla_2 V(R_1, x_1)\, ,\\
R_0^TR_0&=&{\rm Id}_3\, ,   \qquad 
R_1^T R_1={\rm Id}_3\, ,
\end{eqnarray*}
where $P_i$ and $p_i$ are the corresponding momenta for $R_i$ and $x_i$, respectively, for $i=0,1$; $\Lambda$, $\tilde{\Lambda}$ are $(3\times 3)$-symmetric matrices acting as Lagrange multipliers \citep[Chapter 8.1]{LeRe}, 
$\nabla_iV$ is the gradient with respect to the $i$-th family of coordinates.

The RATTLE method is obtained by adding the constraint conditions $$(R_i,x_i; P_i,p_i)\in Leg_L (T(SO(3)\times {\mathbb R}^3))$$ for $i=0, 1$. The Legendre transformation ${\mathbb F} L\colon T({\mathbb R}_{3\times 3}\times {\mathbb R}^3)\rightarrow T^*({\mathbb R}_{3\times 3}\times {\mathbb R}^3)$ is given by 
{
\begin{equation*}
{\mathbb F} L(R,x;\dot{R},\dot{x}) = \left(\dot{R} J, m \dot{x}\right)\,.
\end{equation*}
To introduce the tangency condition,
\begin{equation*}
\mathrm{d}_T \phi(R,\dot{R}) = \dot{R} R^T + R \dot{R}^T = 0,
\end{equation*}
implies that $\dot{R}R^T \in \mathfrak{so}(3)$, where the Lie algebra $\mathfrak{so}(3)$ is the set of skew symmetric matrices. Thus, we need to impose that
\begin{eqnarray*}
 {J}^{-1} P_0^T R_0 + R_0^T P_0 {J}^{-1}&=&0\, , \\  
 {J}^{-1} P_1^T R_1 + R_1^T P_1 {J}^{-1}&=&0 \, .  
\end{eqnarray*}
}
Now taking a basis of the Lie algebra $\mathfrak{so}(3)\simeq T_{\rm e} SO(3)$, that is, the vector space of the skew symmetric matrices: 
\[
e_1=\begin{pmatrix}
0&0&0\\
0&0&-1\\
0&1&0
\end{pmatrix}, \quad 
e_2=\begin{pmatrix}
0&0&1\\
0&0&0\\
-1&0&0
\end{pmatrix},\quad 
e_3=\begin{pmatrix}
0&-1&0\\
1&0&0\\
0&0&0
\end{pmatrix},
\]
Equation~\eqref{eq:pN_pQ} leads to the alternative expression without Lagrange multipliers for the numerical method by definition of the tangent space:
\begin{eqnarray*}
P_0\cdot R_0e_i&=&\frac{R_{1}-R_0}{h}{\mathbf J}\cdot R_0e_i+\frac{h}{2}\nabla_1 V(R_0, x_0)\cdot R_0e_i\, ,\\
p_0&=& m\frac{x_{1}-x_0}{h}+\frac{h}{2}\nabla_2 V(R_0, x_0)\, ,\\
P_1\cdot R_1e_i&=&\frac{R_1-R_{0}}{h}{\mathbf J}\cdot R_1e_i-\frac{h}{2}\nabla_1 V(R_1, x_1)\cdot R_1e_i \, ,\\
p_1&=& m\frac{x_{1}-x_0}{h}-\frac{h}{2}\nabla_2 V(R_1, x_1)\, ,\\
R_0^TR_0&=&{\rm Id}_3\, ,   \qquad 
R_1^T R_1={\rm Id}_3\, .
\end{eqnarray*}
where $A\cdot B=\frac{1}{2}\hbox{tr}(AB^T)$. 

This last example motivates to devote a section to the construction of geometric integrators of constrained systems on a Lie group. 

\end{example}

\section{Symplectic integration of constrained systems on a Lie group}\label{section:Lie}

Lie groups are the configuration manifolds of many  mechanical systems in the field of engineering and robotics. Following Example~\ref{Ex:RigidBody}, we first motivate the need of obtaining sympletic integrators on a Lie group used for multibody mechanical systems, then the symplectic integrators for constrained systems are defined on any Lie group.

\subsection{Motivation: Multibody mechanical systems with holonomic constraints}\label{Sec:multibody}

 For a multibody mechanical system, the total configuration space is typically given by a product of the same Lie group $G$ so that all possible configurations of $s$ rigid bodies can be described, that is, $Q=G^s=\underbrace{G \times G \times \dots \times G}_{s }$. Some typical examples of Lie groups $G$ used in robot manipulators are $SO(2)$, $SE(2)$, $SO(3)$ or $SE(3)$. To describe robot systems is necessary to introduce holonomic constraints representing, for instance, prismatic, revolute joints, etc. Those holonomic constraints naturally restrict  positions and orientations of the rigid bodies to a submanifold $N$ of $Q=G^s$. The process described in Section~\ref{Sec:SymplHolon} is also useful to derive numerical integrators preserving symplecticity according to the steps here delineated for a fixed-step size $h$:

 \begin{enumerate}
    \item Choose an arbitrary discretization map  on the Lie group $G$: $R_d:TG\rightarrow G\times G$. For instance, $R_d(v_g)=(g, g\, \hbox{exp}(g^{-1}v_g))$, where  $\hbox{exp}\colon \mathfrak{g}\simeq T_{\rm Id}G\rightarrow G$ is the exponential map on the Lie group.
\item Consider the natural extension of $R_d$ to the multibody configuration manifold: ${R}^s_d: TG^s\rightarrow G^s\times G^s$.
\item Use the holonomic constraints and the discretization map to define the submanifold ${\mathcal T}_h^{R^s_d}N$ of $TG^s$.

\item Let $L=K-V: TG^s\rightarrow {\mathbb R}$ be a mechanical Lagrangian, construct the Lagrangian submanifold
$\tilde{\Sigma}^d_L$ of the cotangent bundle $ T^*TG^s$
\[
\tilde{\Sigma}^d_L=\left\{\mu\in T^*T G^s\; | \mu-dL\in \left(T\left({\mathcal T}_h^{{R}^s_d}N\right)\right)^0\right\}\,
\] that intrinsically defines the equations of motion for the constrained system, analogous to Equation~\eqref{eq:motionConstrained}.

\item To make computations easier it is possible to define a discretization map on the cotangent bundle given by the cotangent lift $$\Phi^{-1}\circ \widehat{{R}^s_d}: T^*TG^s\rightarrow T^*G^s\times T^*G^s\, ,$$ as appears in Diagram~\eqref{Diagram:cotangentlift} and in Equation~\eqref{new-formula}. This alternative construction for the discretization map will be used again in Section~\ref{section-left}.
\item Finally, define the Lagrangian submanifold {$$(i^*_{N}\times i^*_{N})\left( \left(\Phi^{-1}\circ \widehat{{R}^s_d}\right)(h\tilde{\Sigma}^d_L)\right) $$} of the symplectic manifold $(T^*N\times T^*N, \hbox{pr}_2^*(\omega_N)-\hbox{pr}_1^*(\omega_N))$, where $i_N: N\hookrightarrow G^s$ is the inclusion map.  

\noindent {\bf Caution:} Here, the product  by $h$ in  $h\tilde{\Sigma}^d_L$ is understood with respect to the vector bundle structure $\zeta_{Q}: T^*TQ\rightarrow T^*Q$ (see \cite{grabowska}). Locally,
\[
\zeta_Q(q, \dot{q}, p_q, p_{\dot{q}})=(q, p_{\dot{q}})\; .
\]
Therefore
\begin{equation}\label{double-vector}
h\tilde{\Sigma}^d_L=\{
(q, h\dot{q}, hp_q, p_{\dot{q}})\in T^*TQ\; \; | \;
(q, \dot{q}, p_q, p_{\dot{q}})\in\tilde{\Sigma}^d_L\} 
\end{equation}

\end{enumerate}
Now, the symplectic integrator for the holonomic multibody system is given by taking an initial condition $\mu_{k}$ in $T^*N$ and the next step is to find the unique $\mu_{k+1}\in T^*N$ such that {$$(\mu_k, \mu_{k+1})\in (i^*_{N}\times i^*_{N})\left( \left(\Phi^{-1}\circ \widehat{{R}^s_d}\right)(h\tilde{\Sigma}^d_L)\right) \; .$$} By construction the method is symplectic and preserves exactly the holonomic constraints $N$ on $G^s$. 

{Alternatively, we can follow the extrinsic point of view and consider the implicit symplectic method
\begin{equation*} 
(g^s_0, p^s_0; g^s_1, p^s_1)\; \in \;  \left( \left(\Phi^{-1}\circ \widehat{{R}^s_d}\right)(h\tilde{\Sigma}^d_L)\right)\cap \left(Leg_{{L}}(TG^s) \times Leg_{{L}}(TG^s)\right)\, . 
\end{equation*} 
In Section~\ref{section-left} the left-trivialization is used to give a simplified expression of the previous symplectic methods for a Lie group.
}

\subsection{Constrained symplectic integrators via trivialization}\label{section-left}

Let $G$ be a Lie group with the corresponding Lie algebra ${\mathfrak g}$. Lagrangian mechanics and Hamiltonian mechanics are traditionally expressed on $G\times {\mathfrak g}\equiv TG$ and $ G\times {\mathfrak g}^*\equiv T^*G$, respectively, where we are using left-trivialization (see Appendix~\ref{Sec:Liegroup} for more details on definitions and notations used here).

We first apply the discretization procedure described in~\citep{iser-munt,bou-rabee} to differential equations without constraints on Lie groups by introducing a retraction map 
  $\tau: {\mathfrak g}\to G$ that induces a discretization map $R_d: TG\equiv G\times {\mathfrak g}\rightarrow G\times G$ on $G$ as defined in Section~\ref{Sec:Rd}. For simplicity,  the following discretization map defining the symplectic Euler method-A in Section~\ref{section:euler-a} is used from now on
\[
R_d(g, \xi)=(g, g\, \tau(\xi)),
\]
whose inverse map is $(R_d)^{-1}(g_0, g_1)=\left(g_0, \tau^{-1}\left(g_0^{-1} g_1\right)\right)$.

In the unconstrained case for mechanical systems, 
 the following identifications induced by the left-translation
\begin{eqnarray*}
TTG&\equiv& T(G\times {\mathfrak g})\equiv TG\times T{\mathfrak g}\equiv G\times {\mathfrak g}\times {\mathfrak g}\times {\mathfrak g}\,\\
 T(G\times G) &\equiv& G\times G \times {\mathfrak g}\times  {\mathfrak g}
 \end{eqnarray*}
 allow dualization to define discretization maps on $T^*G\equiv G\times {\mathfrak g}^*$: 
\begin{equation}\label{eq:cotLiftLiegroup}
\widehat{R_d}: 
T^*(G\times {\mathfrak g})\rightarrow 
(G\times {\mathfrak g}^*)\times (G\times {\mathfrak g}^*)\,.
\end{equation}
Note that we are using the same notation as in Equation~\eqref{eq:cotLiftdiscrete} in order not to introduce more cumbersome notation. However, in the current section the domain of the discretization map is a suitable cotangent bundle as the one in Equation~\eqref{eq:cotLiftLiegroup}. 

After some computations, using Equations~\eqref{eq:symplecticPhi} and the trivilizations of the tangent maps in~\eqref{eq:lefttrivialin},~\eqref{eq:righttrivialinv}, it is obtained that
\[
\widehat{R_d} (g, \xi; p_g, p_{\xi})=
\left(g, (d^R\tau^{-1}_{\xi})^*p_{\xi}- {\mathcal L}_{g}^*p_g; g\tau(\xi), (d^L\tau^{-1}_{\xi})^*p_{\xi}\right),
\]
where ${\mathcal L}^*$ is the shorthand notation in Equation~\eqref{eq:AppShortTL*} for the pull-back of the tangent map of the left multiplication.

For a  trivialized  Lagrangian function $L: G\times {\mathfrak g}\rightarrow  {\mathbb R}$ the equations of motion are obtained from Equations~\eqref{new-formula} and~\eqref{double-vector} as:
\begin{equation}(g_0,\alpha_0; g_1, \alpha_1)\in \widehat{R_d}\left( h\, \Sigma^d_{\tilde{L}} \right)\subset (G\times {\mathfrak g}^*)\times (G\times {\mathfrak g}^*), \end{equation}

where the multiplication by $h$ is on the fibers of the tangent bundle of $G\times \mathfrak{g}^*$ and
\[
h\, \Sigma^d_{\tilde{L}}=\left\{(g, h\xi, h\frac{\partial L}{\partial g}(g, \xi), \frac{\partial L}{\partial \xi}(g, \xi)\right\}.
\]

Equivalently, 
\begin{eqnarray*}
g_1&=&g_0 \tau (h\xi)\, ,\\
(\mbox{d}^L\tau_{h\xi})^*(\alpha_1)&=&\frac{\delta L}{\delta \xi}(g_0, \xi)\; ,\\
\alpha_0-\hbox{Ad}^*_{\tau(-h\xi)}\alpha_1&=&-h{\mathcal L}_{g_0}^*\left(\frac{\delta L}{\delta g}(g_0, \xi)\right)\; ,
\end{eqnarray*}
or, alternatively,
\begin{eqnarray*}
g_1&=&g_0 \tau (h\xi)\, ,\\
(\mbox{d}^L\tau_{h\xi})^*(\alpha_1)&=&\frac{\delta L}{\delta \xi}(g_0, \xi)\; ,\\
\alpha_0&=&
(\mbox{d}^R\tau^{-1}_{h\xi})^*\frac{\delta L}{\delta \xi}(g_0, \xi)
-h{\mathcal L}_{g_0}^*\left(\frac{\delta L}{\delta g}(g_0, \xi)\right)\; ,
\end{eqnarray*}
because $\mathrm{Ad}_{\tau(h\xi)} \mathrm{d}^L \tau_{h\xi} = \mathrm{d}^R \tau_{h\xi}$ and, in consequence, ${\rm Ad}^*_{\tau(-h\xi)}=({\rm d}^R\tau^{-1}_{h\xi})^*\circ (\mbox{d}^L\tau_{h\xi})^*$. 

In the holonomic case, we additionally have a constrained submanifold $N\subset G$ which is not in general a Lie subgroup and a regular Lagrangian $L: G\times {\mathfrak g}\rightarrow {\mathbb R}$. As in Section~\ref{Sec:SymplHolon}, we introduce the {\bf modified constrained variational problem}
$(L, {\mathcal T}_h^{R_d}N)$, where
${\mathcal T}_h^{R_d}N=(R_d^h)^{-1}(N\times N)\subset TG\equiv G\times {\mathfrak g}$. If $N$ is defined by constraint functions $\phi^{\alpha}(g)=0$, then $(g,\xi)$ lies in ${\mathcal T}_h^{R_d}N$ if and only if 
\begin{equation*}
\phi^{\alpha}(g)=0\, ,\qquad \phi^{\alpha}(g\tau(h\xi))=0\; ,
\end{equation*}
where $\tau: {\mathfrak g}\rightarrow G$ is a retraction map.

Consequently, for $g_0=g$ the discrete constrained equations~\eqref{Eq:HMethod} and~\eqref{Eq:HMethod2} become the following ones on Lie groups (see Appendix \ref{Sec:Liegroup} for technical details): 
\begin{eqnarray*}
g_1&=&g_0 \tau (h\xi)\, ,\\
(\mbox{d}^L\tau_{h\xi})^*(\alpha_1)&=&\frac{\delta L}{\delta \xi}(g_0, \xi)+h\tilde{\lambda}_{\alpha}((T_{h\xi}\tau)^*\circ {\mathcal L}^*_{g_0})\left(\frac{\delta \phi^{\alpha}}{\delta g}(g_1)\right)\; ,\\
\alpha_0-\hbox{Ad}^*_{\tau(-h\xi)}\alpha_1&=&-\, h{\mathcal L}_{g_0}^*\left(\frac{\delta L}{\delta g}(g_0, \xi)\right)-h{\lambda}_{\alpha} {\mathcal L}_{g_0}^*\left(\frac{\delta \phi^{\alpha}}{\delta g}(g_0)\right)\, \\ &&-h\tilde{\lambda}_{\alpha} ({\mathcal L}_{g_0}\circ {\mathcal R}_{\tau(h\xi)})^*\left(\frac{\delta \phi^{\alpha}}{\delta g}(g_1)\right)
 \; ,\\
 (g_0, \alpha_0), (g_1, \alpha_1)&\in& Leg_L(TN)\, ,
\end{eqnarray*}
where $TN$ is understood as a submanifold of $N\times {\mathfrak g}$ and $Leg_L: G\times {\mathfrak g}\rightarrow G\times {\mathfrak g}^*$ is a diffeomorphism. 
An alternative expression of the symplectic Euler method-A integrator on Lie groups is: 
\begin{eqnarray*}
g_1&=&g_0 \tau (h\xi)\, ,\\
(\mbox{d}^L\tau_{h\xi})^*(\alpha_1)&=&\frac{\delta L}{\delta \xi}(g_0, \xi)+h\tilde{\lambda}_{\alpha}((T_{h\xi}\tau)^*\circ {\mathcal L}^*_{g_0})\left(\frac{\delta \phi^{\alpha}}{\delta g}(g_1)\right)\; ,\\
\alpha_0&=&
(\mbox{d}^R\tau^{-1}_{h\xi})^*\frac{\delta L}{\delta \xi}(g_0, \xi)
-h{\mathcal L}_{g_0}^*\left(\frac{\delta L}{\delta g}(g_0, \xi)\right)-h{\lambda}_{\alpha} {\mathcal L}_{g_0}^*\left(\frac{\delta \phi^{\alpha}}{\delta g}(g_0)\right)\; .
\end{eqnarray*}

{If we apply the following discretization map
\[
R_d(g, \xi)=(g\tau(-\xi), g)
\] coming from the symplectic Euler method-B in Section~\ref{section:euler-b}, the cotangent lift of the discretization map is given by:
\[
\widehat{R_d} (g, \xi; p_g, p_{\xi})=
\left(g\tau(-\xi), (d^L\tau^{-1}_{-\xi})^*p_{\xi}; g, (d^R\tau^{-1}_{-\xi})^*p_{\xi}+ {\mathcal L}_{g_1}^*p_g\right).
\]
}
The constraint functions defining ${\mathcal T}_h^{R_d}N$ are \begin{equation*}
\phi^{\alpha}(g\tau(-h\xi))=0\, ,\qquad \phi^{\alpha}(g)=0\; .
\end{equation*}
As a result, the following method for the constrained system is obtained:
\begin{eqnarray*}
g_1&=&g_0 \tau (h\xi)\, ,\\
(\mbox{d}^L\tau_{-h\xi})^*(\alpha_0)&=&\frac{\delta L}{\delta \xi}(g_1, \xi)-h{\lambda}_{\alpha}((T_{-h\xi}\tau)^*\circ {\mathcal L}^*_{g_1})\left(\frac{\delta \phi^{\alpha}}{\delta g}(g_0)\right)\; ,\\
\alpha_1-\hbox{Ad}^*_{\tau(h\xi)}\alpha_0&=&h{\mathcal L}_{g_1}^*\left(\frac{\delta L}{\delta g}(g_1, \xi)\right)+h{\lambda}_{\alpha} ({\mathcal R}^*_{\tau(-h\xi)}\circ {\mathcal L}^*_{g_1})\left(\frac{\delta \phi^{\alpha}}{\delta g}(g_0)\right)\, ,\\ &&+h\tilde{\lambda}_{\alpha} {\mathcal L}^*_{g_1}\left(\frac{\delta \phi^{\alpha}}{\delta g}(g_1)\right)
 \; ,  \\
 (g_0, \alpha_0), \; (g_1, \alpha_1)&\in& Leg_L(TN)\, ,
\end{eqnarray*}
or, equivalently, 
\begin{eqnarray*}
g_1&=&g_0 \tau (h\xi)\\
(\mbox{d}^L\tau_{-h\xi})^*(\alpha_0)&=&\frac{\delta L}{\delta \xi}(g_1, \xi)-h\lambda_{\alpha}((T_{-h\xi}\tau)^*\circ {\mathcal L}^*_{g_1})\left(\frac{\delta \phi^{\alpha}}{\delta g}(g_0)\right)\; ,\\
\alpha_1&=&
(\mbox{d}^R\tau^{-1}_{-h\xi})^*\frac{\delta L}{\delta \xi}(g_1, \xi)
+h{\mathcal L}_{g_1}^*\left(\frac{\delta L}{\delta g}(g_1, \xi)\right)+h{\tilde{\lambda}}_{\alpha} {\mathcal L}_{g_1}^*\left(\frac{\delta \phi^{\alpha}}{\delta g}(g_1)\right)\; .
\end{eqnarray*}

The composition of both methods gives us the following 2-stage Lobato IIIAB pair, a second order method,
\begin{eqnarray}
g_{1/2}=g_0 \tau \left(\frac{h}{2}\xi\right),\qquad  g_1= g_{1/2}\tau \left(\frac{h}{2}\bar{\xi}\right),\quad && \nonumber \\
\alpha_0=
(\mbox{d}^R\tau^{-1}_{(h/2)\xi})^*\frac{\delta L}{\delta \xi}(g_0, \xi)
-\frac{h}{2}{\mathcal L}_{g_0}^*\left(\frac{\delta L}{\delta g}(g_0, \xi)\right)-\frac{h}{2}{\lambda}_{\alpha} {\mathcal L}_{g_0}^*\left(\frac{\delta \phi^{\alpha}}{\delta g}(g_0)\right)\; .&\label{first-equation}&\\
\alpha_1=
(\mbox{d}^R\tau^{-1}_{-(h/2)\bar{\xi}})^*\frac{\delta L}{\delta \xi}(g_1, \bar{\xi})
+\frac{h}{2}{\mathcal L}_{g_1}^*\left(\frac{\delta L}{\delta g}(g_1, \bar{\xi})\right)+\frac{h}{2}{\tilde{\Lambda}}_{\alpha} {\mathcal L}_{g_1}^*\left(\frac{\delta \phi^{\alpha}}{\delta g}(g_1)\right)\; ,\label{forth-equation} &&\\
(\mbox{d}^L\tau_{(h/2)\xi})^*(\alpha_{1/2})=\frac{\delta L}{\delta \xi}(g_0, \xi)+\frac{h}{2}\tilde{\lambda}_{\alpha}((T_{(h/2)\xi}\tau)^*\circ {\mathcal L}^*_{g_0})\left(\frac{\delta \phi^{\alpha}}{\delta g}(g_{1/2})\right)\; ,&&\label{second-equation}\\
(\mbox{d}^L\tau_{-(h/2)\bar{\xi}})^*(\alpha_{1/2})=\frac{\delta L}{\delta \xi}(g_1, \bar{\xi})-\frac{h}{2}\Lambda_{\alpha}((T_{-(h/2)\bar{\xi}}\tau)^*\circ {\mathcal L}^*_{g_1})\left(\frac{\delta \phi^{\alpha}}{\delta g}(g_{1/2})\right)\; ,&&\label{third-equation}\\
(g_0, \alpha_0), (g_{1/2}, \alpha_{1/2}), (g_1, \alpha_1)\in Leg_L(TN)\ .&& \nonumber
\end{eqnarray}
Given initial conditions $(g_0, \alpha_0)\in Leg_L(TN)\subseteq G\times {\mathfrak g}^*$ we obtain the value $\xi\in {\mathfrak g}$ from Equation (\ref{first-equation}). Now, using Equations (\ref{second-equation}), (\ref{third-equation}) and $(g_{1/2}, \alpha_{1/2})\in Leg_L(TN)$ we obtain the value $\bar{\xi}$. Finally, from Equation (\ref{forth-equation}) and $(g_1, \alpha_1)\in Leg_L(TN)$ we obtain the value $\alpha_1$.

Alternatively, if we first compose the Lagrangian submanifolds and we impose the conditions on the initial and final points that they belong to $Leg_L(TN)$ we obtain an extension of the RATTLE method for Lie groups:
\begin{eqnarray}
g_{1}=g_0 \tau \left(\frac{h}{2}\xi\right)\tau \left(\frac{h}{2}\bar{\xi}\right)\quad , && \nonumber \\
\alpha_0=
(\mbox{d}^R\tau^{-1}_{(h/2)\xi})^*\frac{\delta L}{\delta \xi}(g_0, \xi)
-\frac{h}{2}{\mathcal L}_{g_0}^*\left(\frac{\delta L}{\delta g}(g_0, \xi)\right)-\frac{h}{2}{\lambda}_{\alpha} {\mathcal L}_{g_0}^*\left(\frac{\delta \phi^{\alpha}}{\delta g}(g_0)\right)\; ,&\label{first-equationAlt}&\\
\alpha_1=
(\mbox{d}^R\tau^{-1}_{-(h/2)\bar{\xi}})^*\frac{\delta L}{\delta \xi}(g_1, \bar{\xi})
+\frac{h}{2}{\mathcal L}_{g_1}^*\left(\frac{\delta L}{\delta g}(g_1, \bar{\xi})\right)+\frac{h}{2}{\tilde{\Lambda}}_{\alpha} {\mathcal L}_{g_1}^*\left(\frac{\delta \phi^{\alpha}}{\delta g}(g_1)\right)\; ,\label{forth-equationAlt} &&\\
(\mbox{d}^L\tau^{-1}_{(h/2)\xi})^*\left(\frac{\delta L}{\delta \xi}(g_0, {\xi})\right)=
(\mbox{d}^L\tau^{-1}_{{-}(h/2)\bar{\xi}})^*\left(\frac{\delta L}{\delta \xi}(g_1, \bar{\xi})\right) \label{eq:p0RattleLie}  \\
(g_0, \alpha_0), (g_1, \alpha_1)\in Leg_L(TN)\ .&& \nonumber
\end{eqnarray}

\begin{remark}
    The previous method is different to the method RATTLie proposed in ~\citep{HanteArnoldRattleLie} where the authors use discrete variational calculus for constrained systems on Lie groups. The reason of the non-equivalence is that we are working in general Lie groups and not in standard euclidean spaces. 
\end{remark}

\section{Conclusions and future work }\label{S:future}

In this paper, we have developed an alternative  geometric procedure to construct symplectic methods for constrained mechanical systems using the notion of discretization map introduced in an earlier paper~\citep{21MBLDMdD}. For the particular case of holonomic constraints, the continuous dynamic is used to obtain the geometric integrator because the discretization map focuses on discretizing the tangent bundle of the configuration manifold. As a result, the constraint submanifold is exactly preserved by the discrete flow and the extension of the methods to the case of non-linear configuration spaces is straightforward.

In future work that construction to derive symplectic integrators will be applied to specific and concrete holonomic systems defined on differentiable manifolds, as for instance, multibody system dynamics subjected to holonomic constraints. 

A further extension will consider mechanical systems with nonholonomic constraints and characterize the notion of the discrete null space method in that case.

\section*{Acknowledgements}            
The authors  acknowledge financial support from the Spanish Ministry of Science and Innovation, under grants  PID2022-137909NB-C21, RED2022-134301-TD and   BBVA Foundation via the project “Mathematical optimization for a more efficient, safer and decarbonized maritime transport”.

\appendix

\section{Appendix: Lagrangian and Hamiltonian  systems on Lie groups}  \label{Sec:Liegroup}
Let $G$ be a finite dimensional Lie group. 
The left multiplication ${\mathcal L}_g$ allows us to trivialize the tangent bundle $TG$ and the cotangent bundle $T^*G$ as follows
\begin{eqnarray*}
TG&\to&G\times {\mathfrak g}\, ,\qquad (g, \dot{g})\longmapsto (g, g^{-1}\dot{g})=(g, T_g{\mathcal L}_{g^{-1}}\dot g)=\colon (g, \xi)\; ,\\
T^*G&\to&G\times {\mathfrak g}^*,\qquad  (g, \mu_g)\longmapsto (g, T^*_e {\mathcal L}_g(\mu_g))=\colon (g, \alpha)\; ,
\end{eqnarray*}
where ${\mathfrak g}=T_eG$ is the Lie algebra of $G$ and $e$ is the identity element of $G$. More details can be found in~\citep[Chapter IV.4]{LiMarle}, \citep{Holm-Schmah-Stoica}. We will also use the shorthand notation 
\begin{equation}\label{eq:AppShortTL*}
{\mathcal L}_g^*(\mu_{g'})=(T_{g^{-1}g'}{\mathcal L}_g)^*(\mu_{g'})\in T^*_{g^{-1}g'}G 
\end{equation}
for all $g, g'\in G$.
Having in mind the above identifications by the left trivialization, it is easy to show that the two canonical structures of the cotangent bundle $T^*G$, the Liouville 1-form $\theta_G$ and the canonical symplectic 2-form $\omega_G$, can now be understood as objects on $G\times \mathfrak{g}^*$:
\begin{eqnarray}
(\theta_G)_{(g, \alpha)}(\xi_1, \nu_1)&=&\langle \alpha, \xi_1\rangle\; ,\\
(\omega_G)_{(g, \alpha)}\left( (\xi_1, \nu_1), (\xi_2, \nu_2)\right)&=&-\langle \nu_1, \xi_2\rangle + \langle \nu_2, \xi_1\rangle+\langle\alpha, [\xi_1, \xi_2]\rangle\; ,\label{omega}
\end{eqnarray}
where $(g, \alpha)\in G\times {\mathfrak g}^*$, $\xi_i\in {\mathfrak g}$ and $\nu_i\in {\mathfrak g}^*$, $i=1, 2$. Observe that the elements of $T_{\alpha_g}T^*G\simeq T_{(g,\alpha)} (G\times \mathfrak{g}^*)$ are identified with the pairs $(\xi, \nu)\in {\mathfrak g}\times {\mathfrak g}^*$.

Let $H: T^*G\equiv G\times {\mathfrak g}^*\longrightarrow {\mathbb R}$ be the Hamiltonian function, we compute
\begin{equation}
dH_{(g, \alpha)}(\xi, \nu)=\left\langle  {\mathcal L}_g^*\left(\frac{\delta H}{\delta g}(g, \alpha)\right), \xi\right\rangle+ \left\langle \nu, \frac{\delta H}{\delta \alpha}(g, \alpha)\right\rangle\; , \label{hami-0}
\end{equation}
since $\frac{\delta H}{\delta \alpha}(g, \alpha)\in {\mathfrak g}^{**}={\mathfrak g}$.

Thus, the classical Hamilton's equations whose solutions are integral curves of the Hamiltonian vector field $X_H$ on $T^*G$ can be rewritten as integral curves of a vector field on $G\times \mathfrak{g}^*$. After left-trivialization the equations become $X_H(g, \alpha)=(\xi, \nu)$, where $\xi\in {\mathfrak g}$ and $\nu\in {\mathfrak g}^*$ are elements to be determined using the Hamilton's equations
\[
i_{X_H}\omega_G=dH\; .
\]
 Therefore, from expressions (\ref{omega}) and (\ref{hami-0}), we deduce that
\begin{eqnarray*}
\xi&=&\frac{\delta H}{\delta \alpha}(g, \alpha)\; ,\\
\nu&=&-{\mathcal L}_g^*\left(\frac{\delta H}{\delta g}(g, \alpha)\right)+ad_{\xi}^*\alpha\; .
\end{eqnarray*}
As $\dot{g}=g\xi$, we obtain the Euler-Arnold equations on $G\times \mathfrak{g}^*$:
\begin{eqnarray*}
\dot{g}&=&T_e{\mathcal L}_g\left(\frac{\delta H}{\delta \alpha}(g, \alpha)\right)\equiv g \frac{\delta H}{\delta \alpha}(g, \alpha)\; ,\\
\dot{\alpha}&=&-{\mathcal L}_g^*\left(\frac{\delta H}{\delta g}(g, \alpha)\right)+ad_{\frac{\delta H}{\delta \alpha}(g, \alpha)}^*\alpha\; . 
 \end{eqnarray*}

 A retraction map 
  $\tau: {\mathfrak g}\to G$ is an analytic local
diffeomorphism around the identity such that $\tau(\xi)\tau(-\xi)={e}$ for any $\xi\in\mathfrak g$. Thereby, $\tau$ provides a local chart on the Lie
group. For such a retraction map~\citep{BouMa}, the left trivialized tangent map $\mbox{d}^L\tau_{\xi}:\mathfrak{g}\rightarrow\mathfrak{g}$ and the inverse map $\mbox{d}^L\tau_{\xi}^{-1}:\mathfrak{g}\rightarrow\mathfrak{g}$ are defined for all   $\eta\in\mathfrak{g}$ as follows
  \begin{eqnarray}
    T_{\xi}\tau (\eta)&=&T_{ e} {\mathcal L}_{\tau(\xi)}\left(\mbox{d}^L\tau_{\xi}(\eta)\right)\equiv \tau(\xi)\mbox{d}^L\tau_{\xi}(\eta)\,, \nonumber \\
    T_{\tau(\xi)}\tau^{-1}((T_{ e} {\mathcal L}_{\tau(\xi)})\eta)&=&\mbox{d}^L\tau^{-1}_{\xi}(\eta). \label{eq:lefttrivialin}
  \end{eqnarray}
In other words,
$$ \mbox{d}^L\tau_{\xi}(\eta)\colon = \tau(-\xi) T_{\xi}\tau (\eta) =T_{\tau(\xi)} {\mathcal L}_{\tau(-\xi)}T_{\xi}\tau (\eta) \, .$$
Analogously, it can be defined the right trivialized tangent map $\mbox{d}^R\tau_{\xi}:\mathfrak{g}\rightarrow \mathfrak{g}$ given by 
  \begin{eqnarray}\mbox{d}^R\tau_{\xi}(\eta)&=&T_{\tau(\xi)}{\mathcal R}_{\tau(-\xi)}T_\xi \tau(\eta)=T_\xi \tau(\eta)\,  \tau(-\xi) \, , \nonumber \\ \mbox{d}^R\tau^{-1}_{\xi}(\eta)&=&T_{\tau(\xi)} \tau^{-1} (T_e{\mathcal R}_{\tau(\xi)}(\eta))\,  \label{eq:righttrivialinv}.\end{eqnarray}

\bibliography{references}

\begin{thebibliography}{37}
\providecommand{\natexlab}[1]{#1}
\providecommand{\url}[1]{\texttt{#1}}
\expandafter\ifx\csname urlstyle\endcsname\relax
  \providecommand{\doi}[1]{doi: #1}\else
  \providecommand{\doi}{doi: \begingroup \urlstyle{rm}\Url}\fi

\bibitem[Abraham and Marsden(1978)]{AbMa}
R.\ Abraham and J.E.\ Marsden.
\newblock \emph{Foundations of mechanics}.
\newblock Benjamin/Cummings Publishing Co. Inc. Advanced Book Program, Reading,
  Mass., 1978.
\newblock Second edition, revised and enlarged, With the assistance of Tudor
  Ra{\c{t}}iu and Richard Cushman.

\bibitem[Absil et~al.(2008)Absil, Mahony, and Sepulchre]{AbMaSeBookRetraction}
P.-A. Absil, R.~Mahony, and R.~Sepulchre.
\newblock \emph{Optimization algorithms on matrix manifolds}.
\newblock Princeton University Press, Princeton, NJ, 2008.
\newblock ISBN 978-0-691-13298-3.
\newblock \doi{10.1515/9781400830244}.
\newblock URL \url{https://doi.org/10.1515/9781400830244}.
\newblock With a foreword by Paul Van Dooren.

\bibitem[Adler et~al.(2002)Adler, Dedieu, Margulies, Martens, and
  Shub]{2002Adler}
R.~L. Adler, J.-P. Dedieu, J.~Y. Margulies, M.~Martens, and M.~Shub.
\newblock Newton's method on {R}iemannian manifolds and a geometric model for
  the human spine.
\newblock \emph{IMA J. Numer. Anal.}, 22\penalty0 (3):\penalty0 359--390, 2002.
\newblock ISSN 0272-4979.
\newblock \doi{10.1093/imanum/22.3.359}.
\newblock URL \url{https://doi.org/10.1093/imanum/22.3.359}.

\bibitem[Andersen(1983)]{RattleAndersen}
H.~C. Andersen.
\newblock Rattle: A “velocity” version of the shake algorithm for molecular
  dynamics calculations.
\newblock \emph{Journal of Computational Physics}, 52\penalty0 (1):\penalty0
  24--34, 1983.
\newblock ISSN 0021-9991.
\newblock \doi{https://doi.org/10.1016/0021-9991(83)90014-1}.
\newblock URL
  \url{https://www.sciencedirect.com/science/article/pii/0021999183900141}.

\bibitem[Arnold(1989)]{arnold}
V.~I. Arnold.
\newblock \emph{Mathematical methods of classical mechanics}, volume~60 of
  \emph{Graduate Texts in Mathematics}.
\newblock Springer-Verlag, New York, 1989.
\newblock Translated from the 1974 Russian original by K. Vogtmann and A.
  Weinstein, Corrected reprint of the second (1989) edition.

\bibitem[Barbero-Li\~n\'an and Mart\'in~de Diego(2023)]{21MBLDMdD}
M.~Barbero-Li\~n\'an and D.~Mart\'in~de Diego.
\newblock Retraction maps: a seed of geometric integrators.
\newblock \emph{Found. Comput. Math.}, 23\penalty0 (4):\penalty0 1335--1380,
  2023.
\newblock ISSN 1615-3375,1615-3383.
\newblock \doi{10.1007/s10208-022-09571-x}.
\newblock URL \url{https://doi.org/10.1007/s10208-022-09571-x}.

\bibitem[Betsch(2005)]{Betsch}
P.~Betsch.
\newblock The discrete null space method for the energy consistent integration
  of constrained mechanical systems {P}art {I}: {H}olonomic constraints.
\newblock \emph{Comput. Methods Appl. Mech. Egrg.}, 194:\penalty0 5159--5190,
  2005.
\newblock \doi{10.1016/j.cma.2005.01.004}.
\newblock URL \url{https://doi.org/10.1016/j.cma.2005.01.004}.

\bibitem[Betsch and Leyendecker(2006)]{BetschII}
P.~Betsch and S.~Leyendecker.
\newblock The discrete null space method for the energy consistent integration
  of constrained mechanical systems. {II}. {M}ultibody dynamics.
\newblock \emph{Internat. J. Numer. Methods Engrg.}, 67\penalty0 (4):\penalty0
  499--552, 2006.
\newblock ISSN 0029-5981.
\newblock \doi{10.1002/nme.1639}.
\newblock URL \url{https://doi.org/10.1002/nme.1639}.

\bibitem[Blanes and Casas(2016)]{blanes}
S.~Blanes and F.~Casas.
\newblock \emph{A concise introduction to geometric numerical integration}.
\newblock Monographs and Research Notes in Mathematics. CRC Press, Boca Raton,
  FL, 2016.
\newblock ISBN 978-1-4822-6342-8.

\bibitem[Bogfjellmo and Marthinsen(2016)]{BogMa}
G.~Bogfjellmo and H.~Marthinsen.
\newblock High-order symplectic partitioned {L}ie group methods.
\newblock \emph{Found. Comput. Math.}, 16\penalty0 (2):\penalty0 493--530,
  2016.
\newblock \doi{10.1007/s10208-015-9257-9}.
\newblock URL \url{https://doi.org/10.1007/s10208-015-9257-9}.

\bibitem[Borsuk(1931)]{1931Borsuk}
K.~Borsuk.
\newblock Sur les retractes.
\newblock \emph{Fund. Math.}, 17, 1931.

\bibitem[Bou-Rabee and Marsden(2009{\natexlab{a}})]{BouMa}
N.~Bou-Rabee and J.~E. Marsden.
\newblock Hamilton-{P}ontryagin integrators on {L}ie groups. {I}.
  {I}ntroduction and structure-preserving properties.
\newblock \emph{Found. Comput. Math.}, 9\penalty0 (2):\penalty0 197--219,
  2009{\natexlab{a}}.
\newblock \doi{10.1007/s10208-008-9030-4}.
\newblock URL \url{https://doi.org/10.1007/s10208-008-9030-4}.

\bibitem[Bou-Rabee and Marsden(2009{\natexlab{b}})]{bou-rabee}
N.~Bou-Rabee and J.~E. Marsden.
\newblock Hamilton-{P}ontryagin integrators on {L}ie groups. {I}.
  {I}ntroduction and structure-preserving properties.
\newblock \emph{Found. Comput. Math.}, 9\penalty0 (2):\penalty0 197--219,
  2009{\natexlab{b}}.
\newblock ISSN 1615-3375,1615-3383.
\newblock \doi{10.1007/s10208-008-9030-4}.
\newblock URL \url{https://doi.org/10.1007/s10208-008-9030-4}.

\bibitem[Celledoni et~al.(2022)Celledoni, \c{C}okaj, Leone, Murari, and
  Owren]{2022CelledoniEtAl}
E.~Celledoni, E.~\c{C}okaj, A.~Leone, D.~Murari, and B.~Owren.
\newblock Lie group integrators for mechanical systems.
\newblock \emph{Int. J. Comput. Math.}, 99\penalty0 (1):\penalty0 58--88, 2022.
\newblock ISSN 0020-7160.
\newblock \doi{10.1080/00207160.2021.1966772}.

\bibitem[Chang et~al.(2022)Chang, Perlmutter, and Vankerschaver]{2022Joris}
D.~E. Chang, M.~Perlmutter, and J.~Vankerschaver.
\newblock Feedback integrators for mechanical systems with holonomic
  constraints.
\newblock \emph{Sensors}, 22\penalty0 (17), 2022.
\newblock ISSN 1424-8220.
\newblock \doi{10.3390/s22176487}.

\bibitem[do~Carmo(1992)]{doCarmo}
M.~P. do~Carmo.
\newblock \emph{Riemannian geometry}.
\newblock Mathematics: Theory \& Applications. Birkh\"{a}user Boston, Inc.,
  Boston, MA, 1992.
\newblock ISBN 0-8176-3490-8.
\newblock \doi{10.1007/978-1-4757-2201-7}.
\newblock URL \url{https://doi.org/10.1007/978-1-4757-2201-7}.
\newblock Translated from the second Portuguese edition by Francis Flaherty.

\bibitem[Duruisseaux and Leok(2022)]{Leok2021Constrained}
V.~Duruisseaux and M.~Leok.
\newblock Accelerated optimization on {R}iemannian manifolds via discrete
  constrained variational integrators.
\newblock \emph{J. Nonlinear Sci.}, 32\penalty0 (4):\penalty0 Paper No. 42, 34,
  2022.
\newblock \doi{10.1007/s00332-022-09795-9}.
\newblock URL \url{https://doi.org/10.1007/s00332-022-09795-9}.

\bibitem[Grabowska and Grabowski(2013)]{grabowska}
K.~Grabowska and J.~Grabowski.
\newblock Tulczyjew triples: from statics to field theory.
\newblock \emph{J. Geom. Mech.}, 5\penalty0 (4):\penalty0 445--472, 2013.
\newblock ISSN 1941-4889,1941-4897.
\newblock \doi{10.3934/jgm.2013.5.445}.
\newblock URL \url{https://doi.org/10.3934/jgm.2013.5.445}.

\bibitem[Guillemin and Sternberg(2013)]{13GuiStern}
V.~Guillemin and S.~Sternberg.
\newblock \emph{Semi-classical analysis}.
\newblock International Press, Boston, MA, 2013.
\newblock ISBN 978-1-57146-276-3.

\bibitem[Hairer et~al.(2010)Hairer, Lubich, and Wanner]{hairer}
E.~Hairer, C.~Lubich, and G.~Wanner.
\newblock \emph{Geometric numerical integration}, volume~31 of \emph{Springer
  Series in Computational Mathematics}.
\newblock Springer, Heidelberg, 2010.
\newblock ISBN 978-3-642-05157-9.
\newblock Structure-preserving algorithms for ordinary differential equations,
  Reprint of the second (2006) edition.

\bibitem[Hante and Arnold(2021)]{HanteArnoldRattleLie}
S.~Hante and M.~Arnold.
\newblock R{ATTL}ie: a variational {L}ie group integration scheme for
  constrained mechanical systems.
\newblock \emph{J. Comput. Appl. Math.}, 387:\penalty0 Paper No. 112492, 14,
  2021.
\newblock \doi{10.1016/j.cam.2019.112492}.
\newblock URL \url{https://doi.org/10.1016/j.cam.2019.112492}.

\bibitem[Holm et~al.(2009)Holm, Schmah, and Stoica]{Holm-Schmah-Stoica}
D.~D. Holm, T.~Schmah, and C.~Stoica.
\newblock \emph{Geometric mechanics and symmetry}, volume~12 of \emph{Oxford
  Texts in Applied and Engineering Mathematics}.
\newblock Oxford University Press, Oxford, 2009.
\newblock ISBN 978-0-19-921291-0.
\newblock From finite to infinite dimensions, With solutions to selected
  exercises by David C. P. Ellis.

\bibitem[Iserles et~al.(2000)Iserles, Munthe-Kaas, N{\o}rsett, and
  Zanna]{iser-munt}
A.~Iserles, H.~Z. Munthe-Kaas, S.~P. N{\o}rsett, and A.~Zanna.
\newblock Lie-group methods.
\newblock In \emph{Acta numerica, 2000}, volume~9 of \emph{Acta Numer.}, pages
  215--365. Cambridge Univ. Press, Cambridge, 2000.
\newblock ISBN 0-521-78037-3.
\newblock \doi{10.1017/S0962492900002154}.
\newblock URL \url{https://doi.org/10.1017/S0962492900002154}.

\bibitem[Jay(1996)]{Jay1996}
L.~Jay.
\newblock Symplectic partitioned {R}unge-{K}utta methods for constrained
  {H}amiltonian systems.
\newblock \emph{SIAM J. Numer. Anal.}, 33\penalty0 (1):\penalty0 368--387,
  1996.
\newblock \doi{10.1137/0733019}.
\newblock URL \url{https://doi.org/10.1137/0733019}.

\bibitem[Jay(2015)]{Jay2015}
L.~O. Jay.
\newblock \emph{Lobatto Methods}, pages 817--826.
\newblock Springer Berlin Heidelberg, Berlin, Heidelberg, 2015.
\newblock ISBN 978-3-540-70529-1.
\newblock \doi{10.1007/978-3-540-70529-1_123}.
\newblock URL \url{https://doi.org/10.1007/978-3-540-70529-1_123}.

\bibitem[Johnson and Murphey(2009)]{Danger}
E.~R. Johnson and T.~D. Murphey.
\newblock Dangers of two-point holonomic constraints for variational
  integrators.
\newblock \emph{American Control Conference}, pages 4723--4728--R308, 2009.

\bibitem[Jordan(2018)]{Jo18}
M.~I. Jordan.
\newblock Dynamical symplectic and stochastic perspectives on gradient-based
  optimization.
\newblock In \emph{Proceedings of the {I}nternational {C}ongress of
  {M}athematicians---{R}io de {J}aneiro 2018. {V}ol. {I}. {P}lenary lectures},
  pages 523--549. World Sci. Publ., Hackensack, NJ, 2018.

\bibitem[Leimkuhler and Reich(2004)]{LeRe}
B.~Leimkuhler and S.~Reich.
\newblock \emph{Simulating {H}amiltonian dynamics}, volume~14 of
  \emph{Cambridge Monographs on Applied and Computational Mathematics}.
\newblock Cambridge University Press, Cambridge, 2004.
\newblock ISBN 0-521-77290-7.

\bibitem[Leimkuhler and Skeel(1994)]{LeimSkeel96}
B.~J. Leimkuhler and R.~D. Skeel.
\newblock Symplectic numerical integrators in constrained {H}amiltonian
  systems.
\newblock \emph{J. Comput. Phys.}, 112\penalty0 (1):\penalty0 117--125, 1994.
\newblock ISSN 0021-9991,1090-2716.
\newblock \doi{10.1006/jcph.1994.1085}.
\newblock URL \url{https://doi.org/10.1006/jcph.1994.1085}.

\bibitem[Libermann and Marle(1987)]{LiMarle}
P.\ Libermann and C.-M.\ Marle.
\newblock \emph{Symplectic geometry and analytical mechanics}, volume~35 of
  \emph{Mathematics and its Applications}.
\newblock D. Reidel Publishing Co., Dordrecht, 1987.
\newblock Translated from the French by Bertram Eugene Schwarzbach.

\bibitem[Marsden and West(2001)]{MW_Acta}
J.E.\ Marsden and M.\ West.
\newblock Discrete mechanics and variational integrators.
\newblock \emph{Acta Numer.}, 10:\penalty0 357--514, 2001.
\newblock ISSN 0962-4929.
\newblock \doi{10.1017/S096249290100006X}.
\newblock URL \url{http://dx.doi.org/10.1017/S096249290100006X}.

\bibitem[Reich(1996)]{Reich1}
S.~Reich.
\newblock Symplectic integration of constrained {H}amiltonian systems by
  composition methods.
\newblock \emph{SIAM J. Numer. Anal.}, 33\penalty0 (2):\penalty0 475--491,
  1996.
\newblock ISSN 0036-1429.
\newblock \doi{10.1137/0733025}.
\newblock URL \url{https://doi.org/10.1137/0733025}.

\bibitem[Sanz-Serna and Calvo(1994)]{sanz-serna}
J.~M. Sanz-Serna and M.~P. Calvo.
\newblock \emph{Numerical {H}amiltonian problems}, volume~7 of \emph{Applied
  Mathematics and Mathematical Computation}.
\newblock Chapman \& Hall, London, 1994.
\newblock ISBN 0-412-54290-0.

\bibitem[Shen et~al.(2022)Shen, Tran, and Leok]{2022LeokSOn}
X.~Shen, K.~Tran, and M.~Leok.
\newblock High-order symplectic {L}ie group methods on {$ SO(n) $} using the
  polar decomposition.
\newblock \emph{J. Comput. Dyn.}, 9\penalty0 (4):\penalty0 529--551, 2022.
\newblock ISSN 2158-2491.
\newblock \doi{10.3934/jcd.2022003}.

\bibitem[Tulczyjew(1976{\natexlab{a}})]{Tu}
W.~M.\ Tulczyjew.
\newblock Les sous-vari\'et\'es lagrangiennes et la dynamique lagrangienne.
\newblock \emph{C. R. Acad. Sci. Paris S\'er. A-B}, 283\penalty0 (8):\penalty0
  Av, A675--A678, 1976{\natexlab{a}}.

\bibitem[Tulczyjew(1976{\natexlab{b}})]{TuHamilton}
W.~M.\ Tulczyjew.
\newblock Les sous-vari\'et\'es lagrangiennes et la dynamique hamiltonienne.
\newblock \emph{C. R. Acad. Sci. Paris S\'er. A-B}, 283\penalty0 (8):\penalty0
  Av, 15--18, 1976{\natexlab{b}}.

\bibitem[Tulczyjew and Urba\'{n}ski(1999)]{1999TuUr}
W.~M. Tulczyjew and P.~Urba\'{n}ski.
\newblock A slow and careful {L}egendre transformation for singular
  {L}agrangians.
\newblock \emph{Acta Phys. Polon. B}, 30\penalty0 (10):\penalty0 2909--2978,
  1999.
\newblock ISSN 0587-4254.
\newblock The Infeld Centennial Meeting (Warsaw, 1998).

\end{thebibliography}

\end{document}